\documentclass[oneside,english]{amsart}
\usepackage[T1]{fontenc}
\usepackage[latin9]{inputenc}
\usepackage{color}
\usepackage{enumitem}
\usepackage{amsthm}
\usepackage{amssymb}
\PassOptionsToPackage{normalem}{ulem}
\usepackage{ulem}

\makeatletter

\newcommand{\mathcircumflex}[0]{\mbox{\^{}}}

\numberwithin{equation}{section}
\numberwithin{figure}{section}
\theoremstyle{plain}
\newtheorem{thm}{\protect\theoremname}[section]
  \theoremstyle{definition}
  \newtheorem{defn}[thm]{\protect\definitionname}
  \theoremstyle{plain}
  \newtheorem{fact}[thm]{\protect\factname}
  \theoremstyle{plain}
  \newtheorem{prop}[thm]{\protect\propositionname}
  \theoremstyle{plain}
  \newtheorem{conjecture}[thm]{\protect\conjecturename}
\newcounter{mainthm}
\theoremstyle{plain}

\newtheorem{main_thm}[mainthm]{Main Theorem}
  \theoremstyle{plain}
  \newtheorem{lem}[thm]{\protect\lemmaname}
  \theoremstyle{remark}
  \newtheorem{rem}[thm]{\protect\remarkname}
  \theoremstyle{remark}
  \newtheorem{notation}[thm]{\protect\notationname}
  \theoremstyle{remark}
  \newtheorem{claim}[thm]{\protect\claimname}
 \newlist{casenv}{enumerate}{4}
 \setlist[casenv]{leftmargin=*,align=left,widest={iiii}}
 \setlist[casenv,1]{label={{\itshape\ \casename} \arabic*.},ref=\arabic*}
 \setlist[casenv,2]{label={{\itshape\ \casename} \roman*.},ref=\roman*}
 \setlist[casenv,3]{label={{\itshape\ \casename\ \alph*.}},ref=\alph*}
 \setlist[casenv,4]{label={{\itshape\ \casename} \arabic*.},ref=\arabic*}
\theoremstyle{remark}    
\newtheorem*{claimi*}{Claim I} 
\theoremstyle{remark}    
\newtheorem*{claimii*}{Claim II} 
  \theoremstyle{plain}
  \newtheorem{cor}[thm]{\protect\corollaryname}
  \theoremstyle{plain}
  \newtheorem{assumption}[thm]{\protect\assumptionname}
  \theoremstyle{definition}
  \newtheorem{example}[thm]{\protect\examplename}
  \theoremstyle{remark}
  \newtheorem*{claim*}{\protect\claimname}
\newcounter{constructioncounter}
\theoremstyle{definition}

\newtheorem{construction}[constructioncounter]{Construction}

\usepackage{multicol}
\usepackage{a4wide}
\usepackage{amssymb}
\usepackage{url}
\makeatother

\usepackage{babel}
  \providecommand{\assumptionname}{Assumption}
  \providecommand{\claimname}{Claim}
  \providecommand{\conjecturename}{Conjecture}
  \providecommand{\corollaryname}{Corollary}
  \providecommand{\definitionname}{Definition}
  \providecommand{\examplename}{Example}
  \providecommand{\factname}{Fact}
  \providecommand{\lemmaname}{Lemma}
  \providecommand{\notationname}{Notation}
  \providecommand{\propositionname}{Proposition}
  \providecommand{\remarkname}{Remark}
 \providecommand{\casename}{Case}
\providecommand{\theoremname}{Theorem}

\begin{document}
\global\long\def\P{\mathbf{F}}
\global\long\def\p{\mathbf{f}}
\global\long\def\q{\mathbf{g}}
\global\long\def\SS{\mathcal{P}}
\global\long\def\C{\mathfrak{C}}
\global\long\def\Lim{\operatorname{Lim}}
\global\long\def\Ss{\mathbb{S}}
\global\long\def\cc{\mathbf{c}}
 \global\long\def\mod{\operatorname{mod}}
\global\long\def\pr{\operatorname{pr}}
\global\long\def\image{\operatorname{im}}
\global\long\def\otp{\operatorname{otp}}
\global\long\def\dec{\operatorname{dec}}
\global\long\def\suc{\operatorname{suc}}
\global\long\def\pre{\operatorname{pre}}
\global\long\def\qe{\operatorname{qe}}
 \global\long\def\ind{\operatorname{ind}}
\global\long\def\Nind{\operatorname{Nind}}
\global\long\def\lev{\operatorname{lev}}
\global\long\def\Suc{\operatorname{Suc}}
\global\long\def\HNind{\operatorname{HNind}}
\global\long\def\minb{{\lim}}
\global\long\def\concat{\mathcircumflex}
\global\long\def\cl{\operatorname{cl}}
\global\long\def\tp{\operatorname{tp}}
\global\long\def\id{\operatorname{id}}
\global\long\def\qf{\operatorname{qf}}
\global\long\def\ai{\operatorname{ai}}
\global\long\def\dtp{\operatorname{dtp}}
\global\long\def\acl{\operatorname{acl}}
\global\long\def\nb{\operatorname{nb}}
\global\long\def\limb{{\lim}}
\global\long\def\leftexp#1#2{{\vphantom{#2}}^{#1}{#2}}
\global\long\def\intr{\operatorname{interval}}
\global\long\def\atom{\emph{at}}
\global\long\def\I{\mathfrak{I}}
\global\long\def\uf{\operatorname{uf}}
\global\long\def\ded{\operatorname{ded}}
\global\long\def\Ded{\operatorname{Ded}}
\global\long\def\Df{\operatorname{Df}}
\global\long\def\eq{\operatorname{eq}}
\global\long\def\DfOne{\operatorname{Df}_{\operatorname{iso}}}
\global\long\def\modp#1{\pmod#1}
\global\long\def\sequence#1#2{\left\langle #1\left|\,#2\right.\right\rangle }
\global\long\def\set#1#2{\left\{  #1\left|\,#2\right.\right\}  }
\global\long\def\Diag{\operatorname{Diag}}
\global\long\def\Nn{\mathbb{N}}
\global\long\def\mathrela#1{\mathrel{#1}}
\global\long\def\twiddle{\mathord{\sim}}
\global\long\def\mathordi#1{\mathord{#1}}
\global\long\def\Tt{\mathcal{T}}

\title{A dependent theory with few indiscernibles}

\thanks{Part of the first author's PhD thesis. }

\thanks{The first author was partially supported by SFB grant 878. The second
author would like to thank the Israel Science Foundation for partial
support of this research (Grants no. 710/07 and 1053/11). No. 975
on the second author's list of publications.}

\author{Itay Kaplan and Saharon Shelah}
\begin{abstract}
We give a full solution to the question of existence of indiscernibles
in dependent theories by proving the following theorem: for every
$\theta$ there is a dependent theory $T$ of size $\theta$ such
that for all $\kappa$ and $\delta$, $\kappa\to\left(\delta\right)_{T,1}$
iff $\kappa\to\left(\delta\right)_{\theta}^{<\omega}$. This means
that unless there are good set theoretical reasons, there are large
sets with no indiscernible sequences. 
\end{abstract}
\maketitle

\section{Introduction}

Indiscernible sequences play a very important role in model theory.
Let us recall the definition.
\begin{defn}
Suppose $M$ is some structure, $A\subseteq M$, $\left(I,<\right)$
is some linearly ordered set, and $\alpha$ some ordinal. A sequence
$\bar{a}=\sequence{a_{i}}{i\in I}\in\left(M^{\alpha}\right)^{I}$
is called \emph{indiscernible over $A$} if for all $n<\omega$, every
increasing $n$-tuple from $\bar{a}$ realizes the same type over
$A$. When $A$ is omitted, it is understood that $A=\emptyset$. 
\end{defn}
A very important fact about indiscernible sequences is that they exist
in the following sense:
\begin{fact}
\label{fac:indiscernibles exist}\cite[Lemma 5.1.3]{TentZiegler}
Let $\left(I,<_{I}\right)$, $\left(J,<_{J}\right)$ be infinite linearly
ordered sets, $\alpha$ some ordinal, $M$ a structure and $A\subseteq M$.
Suppose $\bar{b}=\sequence{b_{j}}{j\in J}$ is some sequence of tuples
from $M^{\alpha}$. Then there exists an indiscernible sequence $\bar{a}=\sequence{a_{i}}{i\in I}$
of tuples of length $\alpha$ in some elementary extension $N$ of
$M$ such that:
\begin{itemize}
\item [$\star$]For any $n<\omega$ and formula $\varphi$, if $M\models\varphi\left(b_{j_{0}},\ldots,b_{j_{n-1}}\right)$
for every $j_{0}<_{J}\cdots<_{J}j_{n-1}$ from $J$ then $N\models\varphi\left(a_{i_{0}},\ldots,a_{i_{n-1}}\right)$
for every $i_{0}<_{I}\cdots<_{I}i_{n-1}$ from $I$. 
\end{itemize}
\end{fact}
This is proved using Ramsey and compactness. 

Sometimes, however, we want a stronger result. For instance we may
want that given any set of elements, there is an indiscernible sequence
in it. This gives rise to the following definition:
\begin{defn}
\label{def:arrow}Let $T$ be a complete first order theory, and let
$\C$ be a monster model of $T$ (i.e., a very big saturated model).
For a cardinal $\kappa$, $n\leq\omega$ and an ordinal $\delta$,
the notation $\kappa\to\left(\delta\right)_{T,n}$ means:
\begin{itemize}
\item [$\star$]For every set $A\subseteq\C^{n}$ of size $\kappa$, there
is a non-constant sequence of elements of $A$ of length $\delta$
which is indiscernible.
\end{itemize}
\end{defn}
This definition was suggested by Grossberg and Shelah in \cite[pg. 208, Definition 3.1(2)]{sh:d}
with a slightly different form%
\footnote{The definition there is: $\kappa\to\left(\delta\right)_{T,n}$ if
and only if for each sequence of length $\kappa$ (of $n$-tuples),
there is an indiscernible sub-sequence of length $\delta$. For us
there is no difference because we are dealing with examples where
$\kappa\not\to\left(\mu\right)_{T,n}$. It is also not hard to see
that when $\delta$ is an infinite cardinal these two definitions
are equivalent. %
}. 

As we remarked above, the mere existence of indiscernibles as in Fact
\ref{fac:indiscernibles exist} follows from Ramsey. It is therefore
no surprise that if a cardinal $\lambda$ enjoys a Ramsey-like property,
then for any countable theory $T$ we would have $\lambda\to\left(\omega\right)_{T,n}$. 

For a cardinal $\kappa$, denote by $\left[\kappa\right]^{<\omega}$
the set of all increasing finite sequences of ordinals below $\kappa$. 
\begin{defn}
For cardinals $\kappa,\theta$ and an ordinal $\delta$, the notation
$\kappa\to\left(\delta\right)_{\theta}^{<\omega}$ means:
\begin{itemize}
\item [$\star$]For every function $f:\left[\kappa\right]^{<\omega}\to\theta$
there is a homogeneous sub-sequence of order-type $\delta$ (i.e.,
there exists an increasing sequence $\left\langle \alpha_{i}\left|\, i<\delta\right.\right\rangle \in\leftexp{\delta}{\kappa}$
and $\left\langle c_{n}\left|\, n<\omega\right.\right\rangle \in\leftexp{\omega}{\theta}$
such that $f\left(\alpha_{i_{0}},\ldots,\alpha_{i_{n-1}}\right)=c_{n}$
for every $i_{0}<\cdots<i_{n-1}<\delta$). 
\end{itemize}
\end{defn}
\begin{prop}
\label{prop:easy direction intro} Let $\kappa,\theta$ be cardinals
and $\delta\geq\omega$ a limit ordinal. If $\kappa\to\left(\delta\right)_{\theta}^{<\omega}$
then for every $n\leq\omega$ and every theory $T$ of cardinality
$\left|T\right|\leq\theta$, $\kappa\to\left(\delta\right)_{T,n}$.
\end{prop}
This will be proved below, see Proposition \ref{prop:Easy direction}. 
\begin{defn}
For an ordinal $\alpha$, the Erd\"os cardinal $\kappa\left(\alpha\right)$
is the least non-zero cardinal $\lambda$ such that $\lambda\to\left(\alpha\right)_{2}^{<\omega}$. 
\end{defn}
The cardinal $\kappa\left(\alpha\right)$ may not always exist, indeed,
it depends on the model of ZFC we are in. 
\begin{fact}
\label{fac:kanamori}\cite[Proposition 7.15]{Kanamori} Suppose $\alpha\geq\omega$
is a limit ordinal, then:
\begin{enumerate}
\item For any $\gamma<\kappa\left(\alpha\right)$, $\kappa\left(\alpha\right)\to\left(\alpha\right)_{\gamma}^{<\omega}$.
\item $\kappa\left(\alpha\right)$ is an uncountable strongly inaccessible
cardinal. 
\end{enumerate}
\end{fact}
In \cite[pg. 209]{sh:d} it is proved that there is a countable simple
unstable theory such that for a limit ordinal $\delta\geq\omega$,
if $\kappa\to\left(\delta\right)_{T,1}$ then $\kappa\to\left(\delta\right)_{2}^{<\omega}$.
It is also very easy to find such a theory with the property that
if $\kappa\to\left(\delta\right)_{T,1}$ then $\kappa\to\left(\delta\right)_{\omega}^{<\omega}$
($T$ would be the model completion of the empty theory in the language
$\set{R_{n,m}}{n,m<\omega}$ where $R_{n,m}$ is an $n$-ary relation).

There it is conjectured that in dependent (NIP) theories (see Definition
\ref{def:nip} below), such phenomenon cannot happen:
\begin{conjecture}
\label{conj:existence of indiscernibles}\cite[pg. 209, Conjecture 3.3]{sh:d}
If $T$ is dependent, then for every cardinal $\mu$ there is some
cardinal $\lambda$ such that $\lambda\to\left(\mu\right)_{T,1}$.
\end{conjecture}
By Proposition \ref{prop:easy direction intro}, if $\kappa\left(\mu\right)$
exists then Conjecture \ref{conj:existence of indiscernibles} holds
for $\mu$ and every theory $T$ (regardless of NIP) with $\left|T\right|<\kappa\left(\mu\right)$.

In stable theories, Conjecture \ref{conj:existence of indiscernibles}
holds in any model of ZFC: 
\begin{fact}
For any $\lambda$ satisfying $\lambda=\lambda^{\left|T\right|}$,
$\lambda^{+}\to\left(\lambda^{+}\right)_{T,n}$. 
\end{fact}
This was proved by Shelah (See \cite{Sh:c}), and follows from local
character of non-forking. 

Conjecture \ref{conj:existence of indiscernibles} also holds in strongly
dependent%
\footnote{For more on strongly dependent theories, see Section \ref{sec:strongly dependent}. %
} theories: 
\begin{fact}
\label{fac:TStrongly} \cite{Sh863} If $T$ is strongly dependent,
then for all $\lambda\geq\left|T\right|$, $\beth_{\left|T\right|^{+}}\left(\lambda\right)\to\left(\lambda^{+}\right)_{T,n}$
for all $n<\omega$.
\end{fact}
Conjecture \ref{conj:existence of indiscernibles} is connected to
a result by Shelah and Cohen: in \cite{ShCo919}, they proved that
a theory is stable iff it can be presented in some sense in a free
algebra\textcolor{black}{{} with a fixed vocabulary, allowing function
symbols with infinite arity.} If this result could be extended to
saying that a theory is dependent iff it can be represented as an
algebra with ordering, then this could be used to prove Conjecture
\ref{conj:existence of indiscernibles}. 

In the previous paper \cite[Theorem 2.11]{KaSh946}, we have shown
that:
\begin{thm}
\label{thm:MainThm946}There exists a countable dependent theory $T$
such that:

For any two cardinals $\mu\leq\kappa$ with no uncountable strongly
inaccessible cardinals in $\left[\mu,\kappa\right]$, $\kappa\not\to\left(\mu\right)_{T,1}$.
\end{thm}
Thus, if $V$ is a model of ZFC without strongly inaccessible cardinals,
then Conjecture \ref{conj:existence of indiscernibles} fails in $V$
(so this conjecture is false in general). Still, one might hope that
this is the only restriction. However, we show that in fact one needs
Erd\"os cardinals to exist. Namely, we show that there is a dependent
theory, of any given cardinality, such that the only reason for which
Conjecture \ref{conj:existence of indiscernibles} could hold for
it is Proposition \ref{prop:easy direction intro}, thus getting the
best possible result. 
\begin{main_thm}
\label{mainthmA}For every $\theta$ there is a dependent theory $T$
of size $\theta$ such that for all cardinals $\kappa$ and limit
ordinals $\delta\geq\omega$, \textup{$\kappa\to\left(\delta\right)_{T,1}$
iff $\kappa\to\left(\delta\right)_{\theta}^{<\omega}$. }
\end{main_thm}
Note that by Fact \ref{fac:kanamori}, Main Theorem \ref{mainthmA}
is a generalization of Theorem \ref{thm:MainThm946}. 

It was unknown to us that in 2011 Kuda{\u\i}bergenov proved a related
result, which refutes a strong version of Conjecture \ref{conj:existence of indiscernibles},
namely that $\beth_{\omega+\omega}\left(\mu+\left|T\right|\right)\to\left(\mu\right)_{T,1}$.
He proved that for every ordinal $\alpha$ there exists a dependent
theory (we have not checked whether it is strongly dependent) $T_{\alpha}$
such that $\left|T_{\alpha}\right|=\left|\alpha\right|+\aleph_{0}$
and $\beth_{\alpha}\left(\left|T_{\alpha}\right|\right)\not\to\left(\aleph_{0}\right)_{T_{\alpha},1}$
and thus seem to indicate that the bound in Fact \ref{fac:TStrongly}
is tight. See \cite{russianConjecture}. 

We would like to thank the anonymous referee for his very careful
reading and many useful remarks.

\subsection{The idea of the proof}

The theory $T$ is a ``tree of trees'' with functions between the
trees. More precisely, for all $\eta$ in the base tree $\Ss=2^{<\omega}$
we have a unary predicate $P_{\eta}$ and an ordering $<_{\eta}$
such that $\left(P_{\eta},<_{\eta}\right)$ is a discrete tree. In
addition we will have functions $G_{\eta,\eta\concat\left\{ i\right\} }:P_{\eta}\to P_{\eta\concat\left\langle i\right\rangle }$
for $i=0,1$. The idea is to prove that if $\kappa\not\to\left(\delta\right)_{\theta}^{<\omega}$
then $\kappa\not\to\left(\delta\right)_{T,1}$ by induction on $\kappa$,
i.e., to prove that we can find a subset of $P_{\left\langle \right\rangle }$
of size $\kappa$ without an indiscernible sequence in it. For $\kappa$
regular but not strongly inaccessible or $\kappa$ singular the proof
is similar to the one in \cite{KaSh946}: we just push our previous
examples into deeper levels.

The main case is when $\kappa$ is strongly inaccessible.

We have a function $\cc$ that witnesses that $\kappa\not\to\left(\delta\right)_{\theta}^{<\omega}$
and we build a model $M_{\cc}$. In this model, the base tree will
be $\omega$ and not $2^{<\omega}$, i.e., for each $n<\omega$ we
have a predicate $P_{n}$ with tree-ordering $<_{n}$ and functions
$G_{n}:P_{n}\to P_{n+1}$. In addition, $P_{0}\subseteq\kappa$. On
$P_{n}$ we will define an equivalence relation $E_{n}$ refining
the neighboring relation ($x,y$ are neighbors if they succeed the
same element) so that every class of neighbors (neighborhood) is a
disjoint union of less than $\kappa$ many classes of $E_{n}$. We
will prove that if there are indiscernibles in $P_{0}$, then there
is some $n<\omega$ such that in $P_{n}$ we get an indiscernible
sequence $\left\langle t_{i}\left|\, i<\delta\right.\right\rangle $
that looks like a fan, i.e., there is some $u$ such that $t_{i}\wedge t_{j}=u$
and $t_{i}$ is the successor of $u$, and in addition $t_{i}$ and
$t_{j}$ are not $E_{n}$ equivalent for $i\neq j$. 

Now embed $M_{\cc}$ into a model of our theory (i.e., now the base
tree is again $2^{<\omega}$), and in each neighborhood we send every
$E_{n}$ class to an element from the model we get from the induction
hypothesis (as there are less than $\kappa$ many classes, this is
possible).

By induction, we get there is no indiscernible sequence in $P_{0}$
and finish.

\subsection{Description of the paper}

In Section \ref{sec:Preliminaries} we give some preliminaries on
dependent and strongly dependent theories and trees. In Section \ref{sec:Construction}
we describe the theory and prove quantifier elimination and dependence.
In Section \ref{sec:The-inaccessible-case} we deal with the main
technical obstacle, namely the inaccessible case. 

In Section \ref{sec:The-main-theorem} we prove the main theorem.
In Section \ref{sec:strongly dependent} we give a parallel result
for $\omega$-tuples in strongly dependent theories.

\section{\label{sec:Preliminaries}Preliminaries}

\subsection*{Notation}

We use standard notation. $a,b,c$ are elements, and $\bar{a},\bar{b},\bar{c}$
are finite or infinite tuples of elements. 

$\C$ will be the monster model of the theory (i.e., a very big, saturated
model).

For a set $A\subseteq\C$, $S_{n}\left(A\right)$ is the set of complete
$n$-types over $A$, and $S_{n}^{\qf}\left(A\right)$ is the set
of all quantifier free complete $n$-types over $A$. For a finite
set of formulas with a partition of variables, $\Delta\left(\bar{x};\bar{y}\right)$,
$S_{\Delta\left(\bar{x};\bar{y}\right)}\left(A\right)$ is the set
of all $\Delta$-types over $A$, i.e., maximal consistent subsets
of $\left\{ \varphi\left(\bar{x},\bar{a}\right),\neg\varphi\left(\bar{x},\bar{a}\right)\left|\,\varphi\left(\bar{x},\bar{y}\right)\in\Delta\,\&\,\bar{a}\in A^{\lg\left(\bar{y}\right)}\right.\right\} $.
Similarly we define $\tp_{\Delta\left(\bar{x};\bar{y}\right)}\left(\bar{b}/A\right)$
as the set of formulas $\varphi\left(\bar{x},\bar{a}\right)$ such
that $\varphi\left(\bar{x},\bar{y}\right)\in\Delta$ and $\C\models\varphi\left(\bar{b},\bar{a}\right)$. 

When $\alpha$ and $\beta$ are ordinals, we use left exponentiation
$\leftexp{\beta}{\alpha}$ to denote the set of functions from $\beta$
to $\alpha$, as to not to confuse with ordinal (or cardinal) exponentiation.
If there is no room for confusion, and $A$ and $B$ are some sets
we use $A^{B}$ instead. The set $\alpha^{<\beta}$ is the set of
sequences (functions) $\bigcup\set{\leftexp{\gamma}{\alpha}}{\gamma<\beta}$.
Similarly, for a set $A$, $A^{<\beta}=\bigcup\set{A^{\gamma}}{\gamma<\beta}$. 

For a sequence $\bar{s}$ (finite or infinite), we denote by $\lg\left(\bar{s}\right)$
its length. If $f$ is a function from some ordinal $\alpha$, then
$\lg\left(f\right)=\alpha$.

\subsection*{Dependent theories}

For completeness, we give here the definitions and basic facts we
need on dependent theories.
\begin{defn}
\label{def:nip}A first order theory $T$ is \emph{dependent} (sometimes
also NIP) if it does not have the independence property: there is
no formula $\varphi\left(\bar{x},\bar{y}\right)$ and tuples $\left\langle \bar{a}_{i},\bar{b}_{s}\left|\, i<\omega,s\subseteq\omega\right.\right\rangle $
from $\C$ such that $\C\models\varphi\left(\bar{a}_{i},\bar{b}_{s}\right)$
iff $i\in s$.

We recall the following fact, which is a consequence of both the so-called
Sauer-Shelah lemma (apparently first proved by Vapnik and Chervonenkis,
then rediscovered by Sauer and again by Shelah in the model theoretic
setting, more or less at the same time) and the fact that if a theory
has the independence property then there is a formula $\varphi\left(x,\bar{y}\right)$
with $\lg\left(x\right)=1$ that witnesses this:\end{defn}
\begin{fact}
\label{fac:DepPolBd} \cite[II, 4]{Sh:c} Let $T$ be any theory.
Then for all $n<\omega$, $T$ is dependent if and only if $\square_{n}$
if and only if $\square_{1}$ where for all $n<\omega$,
\begin{itemize}
\item [$\square_n$] For every finite set of formulas $\Delta\left(\bar{x},\bar{y}\right)$
with $n=\lg\left(\bar{x}\right)$, there is a polynomial $f$ over
$\Nn$ such that for every finite set $A\subseteq M\models T$, $\left|S_{\Delta}\left(A\right)\right|\leq f\left(\left|A\right|\right)$.
\end{itemize}
\end{fact}

\subsection*{Strongly dependent theories}

In \cite{Sh863,Sh783}, the author asks what is a possible solution
to the equation dependent / $x$ = stable / superstable. There, he
discusses several possible strengthening of NIP, namely strongly$^{l}$
dependent theories for $l=1,2,3,4$. These are subclasses of dependent
theories and each one refines the previous one. Strongly$^{1}$ dependent
theories are usually just called strongly dependent, and strongly$^{2}$
theories are sometimes called strongly$^{+}$ theories. These two
classes and related notions (such as dp-rank) were studied much more
than the other two, so we will not mention strongly$^{3}$ or strongly$^{4}$
dependent theories. For instance, strongly$^{2}$ dependent groups
are discussed in \cite{Sh993}. The theories of the reals and the
$p$-adics are both strongly dependent, but neither is strongly$^{2}$
dependent. 

Here are the definitions:
\begin{defn}
A theory $T$ is said to be \uline{not}\emph{ strongly dependent}
if there exists a sequence of formulas $\left\langle \varphi_{i}\left(\bar{x},\bar{y}_{i}\right)\right\rangle $
(where $\bar{x},\bar{y}_{i}$ are tuples of variables), an array $\left\langle \bar{a}_{i,j}\left|\, i,j<\omega\right.\right\rangle $
in $\C$ (where $\lg\left(\bar{a}_{i,j}\right)=\lg\left(\bar{y}_{i}\right)$)
and tuples $\left\langle \bar{b}_{\eta}\left|\,\eta:\omega\to\omega\right.\right\rangle $
($\lg\left(\bar{b}_{\eta}\right)=\lg\left(\bar{x}\right)$) in $\C$
such that $\models\varphi_{i}\left(\bar{b}_{\eta},\bar{a}_{i,j}\right)\Leftrightarrow\eta\left(i\right)=j$.
\end{defn}

\begin{defn}
\label{def:strongly2 dependent}A theory $T$ is said to be \uline{not}\emph{
strongly$^{2}$ dependent} if there exists a sequence of formulas
$\left\langle \varphi_{i}\left(\bar{x},\bar{y}_{i},\bar{y}_{i-1},\ldots,\bar{y}_{0}\right)\left|\, i<\omega\right.\right\rangle $,
an array $\left\langle \bar{a}_{i,j}\left|\, i,j<\omega\right.\right\rangle $
in $\C$ (where $\lg\left(\bar{a}_{i,j}\right)=\lg\left(\bar{y}_{i}\right)$)
and tuples $\left\langle \bar{b}_{\eta}\left|\,\eta:\omega\to\omega\right.\right\rangle $
($\lg\left(\bar{b}_{\eta}\right)=\lg\left(\bar{x}\right)$) in $\C$
such that $\models\varphi_{i}\left(\bar{b}_{\eta},\bar{a}_{i,j},\bar{a}_{i-1,\eta\left(i-1\right)},\ldots,\bar{a}_{0,\eta\left(0\right)}\right)\Leftrightarrow\eta\left(i\right)=j$.
\end{defn}
See \cite[Claim 2.9]{Sh863} for more details. 

We will use the following criterion:
\begin{lem}
\label{lem:strong dep. pol bound} Suppose $T$ is a theory such that
for every number $n<\omega$ there exists some number $N_{n}<\omega$
such that for every finite set of formulas $\Delta\left(\bar{x},\bar{y}\right)$
with $n=\lg\left(\bar{x}\right)$, there is a polynomial $f$ over
$\Nn$ of degree $\leq N_{n}$ such that for every finite set $A\subseteq M\models T$,
$\left|S_{\Delta}\left(A\right)\right|\leq f\left(\left|A\right|\right)$.
Then $T$ is strongly$^{2}$ dependent.\end{lem}
\begin{proof}
Suppose not, then by Definition \ref{def:strongly2 dependent}, we
have a sequence of formulas $\left\langle \varphi_{i}\left(\bar{x},\bar{y}_{i},\bar{y}_{i-1},\ldots,\bar{y}_{0}\right)\left|\, i<\omega\right.\right\rangle $
and an array $\sequence{\bar{a}_{i,j}}{i,j<\omega}$. Suppose $N=N_{\lg\left(\bar{x}\right)}<K<\omega$.
Let $l$ be a bound on $\lg\left(\bar{a}_{i,j}\right)$ for $i<K$,
and for $j<\omega$ let $A_{j}=\bigcup\set{\bigcup\bar{a}_{i,j'}}{i<K,j'<j}$.
Let $\Delta\left(\bar{x};\bar{y}\right)=\set{\varphi_{i}\left(\bar{x},\bar{y}_{i},\ldots,\bar{y}_{0}\right)}{i<K}$.
So the number of $\Delta$-types over $A_{j}$ is at least $j^{K}$
(as the number of functions $\eta:K\to j$). By assumption, $\left|S_{\Delta}\left(A_{j}\right)\right|\leq c\cdot\left|A_{j}\right|^{N}\leq c\cdot\left(l\cdot j\cdot K\right)^{N}$
for some $c\in\Nn$. But for big enough $j$, $c\cdot\left(l\cdot j\cdot K\right)^{N}<j^{K}$
--- contradiction. 
\end{proof}

\subsection*{Trees}

Let us remind the reader of the basic definitions and properties of
trees.
\begin{defn}
A \emph{tree} is a partially ordered set $\left(A,<\right)$ such
that for all $a\in A$, the set $A_{<a}=\left\{ x\left|\, x<a\right.\right\} $
is linearly ordered.
\end{defn}

\begin{defn}
We say that a tree $A$ is \emph{well ordered} if $A_{<a}$ is well
ordered for every $a\in A$. Assume now that $A$ is well ordered.
\begin{itemize}
\item For every $a\in A$, denote $\lev\left(a\right)=\otp\left(A_{<a}\right)$
--- the \emph{level} of $a$ is the order type of $A_{<a}$.
\item The \emph{height} of $A$ is $\sup\left\{ \lev\left(a\right)\left|\, a\in A\right.\right\} $
.
\item $a\in A$ is a \emph{root} if it is minimal.
\item $A$ is \emph{normal} when for all limit ordinals $\delta$, and for
all $a,b\in A$, if 1) $\lev\left(a\right)=\lev\left(b\right)=\delta$,
and 2) $A_{<a}=A_{<b}$, then $a=b$.
\item If $a<b$ then we denote by $\suc\left(a,b\right)$ the \emph{successor}
of $a$ in the direction of $b$, i.e., $\min\left\{ c\leq b\left|\, a<c\right.\right\} $.
\item We write $a<_{\suc}b$ if $b=\suc\left(a,b\right)$.
\item We call $A$ \emph{standard} if it is well ordered, normal, and has
a root.
\end{itemize}
\end{defn}
For a standard tree $\left(A,<\right)$, define $a\wedge b=\max\left\{ c\left|\, c\leq a\,\&\, c\leq b\right.\right\} $.

\section{\label{sec:Construction}Construction of the theory}

In this section we shall introduce the theory $T_{S}$, attached to
a standard tree $S$. Then, for $\Ss=2^{<\omega}$, this theory (or
a variant of it, given by adding constants) will be the theory that
will exemplify Main Theorem \ref{mainthmA}. 

In the first part, we construct the theory $T_{S}^{\forall}$ which
is universal (i.e., all its axioms are of the form $\forall x\varphi$
where $\varphi$ is quantifier free). As we said in the introduction,
the idea is that for every $\eta\in S$, we have a predicate $P_{\eta}$,
and whenever $\eta_{1}<_{\suc}\eta_{2}$ there is a function from
$P_{\eta_{2}}$ to $P_{\eta_{1}}$. Then we would like to take a model
completion $T_{S}$ of this theory (see below). If we put no further
restriction on the theory $T_{S}^{\forall}$, this is easily done
(using AP and JEP, see below), and the model completion will be the
theory of dense trees with functions (if $S$ is a finite tree, then
it is even $\omega$-categorical). This is what we did in \cite[Theorem 2.11]{KaSh946},
but this does not seem to suffice to deal with inaccessible cardinals.
For that reason we further complicate the theory by making the trees
discrete, adding successors and predecessors. This require some constraint
on the functions involved --- ``regressiveness'' --- which is needed
for quantifier elimination. 

Recall that for a given (first order) theory $T$ in a language $L$,
a \emph{model companion} of $T$ is another theory $T'$ in $L$ such
that every model of $T$ can be embedded in a model of $T'$ and vice
versa and in addition $T'$ is \emph{model complete}, i.e., if $M_{1}$
is a substructure of $M_{2}$ and $M_{1},M_{2}\models T'$ then $M_{1}$
is an elementary substructure of $M_{2}$. A model companion of a
theory is unique if it exists. A model companion $T'$ is called a\emph{
model completion} when for every model $M$ of $T$, $T'\cup\Diag_{\qf}\left(M\right)$
is complete ($\Diag_{\qf}\left(M\right)$ is the theory in the language
$L\cup\set{c_{a}}{a\in M}$ that contains all atomic formulas that
hold in $M$). If $T'$ is a model completion of $T$ and $T$ is
universal, then $T'$ eliminates quantifiers for non-sentences. If
in addition $T$ has JEP (see below) then $T'$ is complete. For more,
see e.g., \cite{Hod}. 

A theory $T$ has the \emph{joint embedding property (JEP)} if given
any two models $A$, $B$ of $T$, there is a model $C$ and embeddings
$f:A\to C$, $g:B\to C$.

A theory $T$ has the \emph{amalgamation property (AP)} if given any
three models $A,B$ and $C$ of $T$, and embeddings $f:A\to B$,
$g:A\to C$, there is a model $D$ and embeddings $h:B\to D$, $i:C\to D$
such that $h\circ f=i\circ g$. 

By \cite[Theorem 7.4.1]{Hod}, if a universal theory $T$ in a finite
language is uniformly locally finite (i.e., there is a function $f:\omega\to\omega$
such that for all $M\models T$ and finite $A\subseteq M$, the size
of the structure generated by $A$ is $f\left(\left|A\right|\right)$)
and has AP and JEP, then it has a model completion $T'$ which is
also $\omega$-categorical (this is related to Fra\"iss\'e limits).
In \cite[Theorem 2.11]{KaSh946} we used exactly this criterion to
construct the model completion. Here, however, substructures are not
finite (since we have the successor function), so we cannot apply
this theorem. 

Instead, we show that the the class of existentially closed models
of $T_{S}^{\forall}$ is elementary (recall that a model $M$ of a
theory $T$ is an \emph{existentially closed model of $T$} if for
any extension $N\supseteq M$ such that $N\models T$, every quantifier
free formula $\varphi\left(x\right)$ over $M$ that has a realization
in $N$ has one in $M$). In fact we show that every two existentially
closed models of $T_{S}^{\forall}$ are elementary equivalent (this
uses the fact that $T_{S}^{\forall}$ has JEP). We call their theory
$T_{S}$. In the process we show that $T_{S}$ also eliminates quantifiers.
Thus, this is the model completion of $T_{S}^{\forall}$. 

In the second part, we show that $T_{S}$ is dependent, and that if
$S$ is finite then it is strongly$^{2}$ dependent (using Lemma \ref{lem:strong dep. pol bound}
and quantifier elimination).

Finally we add constants to the language so that its cardinality will
be $\theta$, and call the resulting theory $T_{S}^{\theta}$.

\subsection*{The first order theory}

\uline{The language:}

Let $S$ be a standard tree, and let $L_{S}$ be the language:
\[
\left\{ P_{\eta},<_{\eta},\wedge_{\eta},G_{\eta_{1},\eta_{2}},\suc_{\eta},\pre_{\eta},\limb_{\eta}\left|\,\eta,\eta_{1},\eta_{2}\in S,\eta_{1}<_{\suc}\eta_{2}\right.\right\} .
\]
Where:

$P_{\eta}$ is a unary predicate, $<_{\eta}$ is a binary relation
symbol, $\wedge_{\eta}$ and $\suc_{\eta}$ are binary function symbols,
$G_{\eta_{1},\eta_{2}}$, $\pre_{\eta}$ and $\limb_{\eta}$ are unary
function symbols.
\begin{defn}
Let $L_{S}'=L_{S}\backslash\set{\pre_{\eta},\suc_{\eta}}{\eta\in S}$. 
\end{defn}
\uline{The theory:}
\begin{defn}
\label{def:T_S} The theory $T_{S}^{\forall}$ says:
\begin{itemize}
\item $\left(P_{\eta},<_{\eta}\right)$ is a tree.
\item $\eta_{1}\neq\eta_{2}\Rightarrow P_{\eta_{1}}\cap P_{\eta_{2}}=\emptyset$.
\item $\wedge_{\eta}$ is the meet function: $x\wedge_{\eta}y=\max\left\{ z\in P_{\eta}\left|\, z\leq_{\eta}x\,\&\, z\leq_{\eta}y\right.\right\} $
for $x,y\in P_{\eta}$ (so its existence is part of the theory).
\item $\suc_{\eta}$ is the successor function --- for $x,y\in P_{\eta}$
with $x<_{\eta}y$, $\suc_{\eta}\left(x,y\right)$ is the successor
of $x$ in the direction of $y$. The axioms are:

\begin{itemize}
\item $\forall x<_{\eta}y\left(x<_{\eta}\suc_{\eta}\left(x,y\right)\leq_{\eta}y\right)$,
and
\item $\forall x\leq_{\eta}z\leq_{\eta}\suc_{\eta}\left(x,y\right)\left[z=x\vee z=\suc_{\eta}\left(x,y\right)\right]$.
\end{itemize}
\item $\limb_{\eta}\left(x\right)$ is the greatest limit element below
$x$. Formally,

\begin{itemize}
\item $\minb_{\eta}:P_{\eta}\to P_{\eta}$, $\forall x\limb_{\eta}\left(x\right)\leq_{\eta}x$,
$\forall x<_{\eta}y\left(\limb_{\eta}\left(x\right)\leq_{\eta}\limb_{\eta}\left(y\right)\right)$,
\item $\forall x<_{\eta}y\left(\limb_{\eta}\left(\suc_{\eta}\left(x,y\right)\right)=\limb_{\eta}\left(x\right)\right)$,
$\forall x\limb_{\eta}\left(\limb_{\eta}\left(x\right)\right)=\limb_{\eta}\left(x\right)$.
\end{itemize}
\item Let the successor elements be those $x$'s such that $\limb_{\eta}\left(x\right)<_{\eta}x$,
and denote 
\[
\Suc\left(P_{\eta}\right)=\left\{ x\in P_{\eta}\left|\,\limb_{\eta}\left(x\right)<_{\eta}x\right.\right\} .
\]

\item $\pre_{\eta}$ is the immediate predecessor function from $\Suc\left(P_{\eta}\right)$
to $P_{\eta}$ ---

$\forall x\neq\minb_{\eta}\left(x\right)\left(\pre_{\eta}\left(x\right)<x\wedge\suc_{\eta}\left(\pre_{\eta}\left(x\right),x\right)=x\right)$.

\item (\textbf{regressiveness}) If $\eta_{1}<_{\suc}\eta_{2}$ then $G_{\eta_{1},\eta_{2}}$
satisfies: $G_{\eta_{1},\eta_{2}}:\Suc\left(P_{\eta_{1}}\right)\to P_{\eta_{2}}$
and if $x<_{\eta_{1}}y$, both $x$ and $y$ are successors, and $\limb_{\eta}\left(x\right)=\limb_{\eta}\left(y\right)$,
then $G_{\eta_{1},\eta_{2}}\left(x\right)=G_{\eta_{1},\eta_{2}}\left(y\right)$. 
\item In all the axioms above, for elements or pairs outside of the domain
of any of the functions $\wedge_{\eta},\lim_{\eta},G_{\eta_{1},\eta_{2}},\suc_{\eta}$
or $\pre_{\eta}$, these functions are the identity on the leftmost
coordinate, so for example if $\left(x,y\right)\notin P_{\eta}^{2}$,
then $x\wedge_{\eta}y=x$. 
\end{itemize}
\end{defn}
\begin{rem}
We need the regressiveness axiom so that $T_{S}^{\forall}$ would
have a model completion. Indeed, suppose $S=\left\{ 0,1\right\} $
and we remove this axiom, and suppose that $T$ is a model completion
of $T_{S}^{\forall}$. Then every model of $T$ is an existentially
closed model of $T_{S}^{\forall}$. Suppose $M\models T$ and $a<_{0}^{M}b\in\Suc\left(P_{0}^{M}\right)$.
Then if $b$ is greater than $\suc_{0}^{M}\left(\cdots\left(\suc_{0}^{M}\left(a,b\right)\right)\right)$
for every finite number of compositions then there is some $a<_{0}^{M}c<_{0}^{M}b$
in $M$ such that $G_{0,1}^{M}\left(c\right)\neq G_{0,1}^{M}\left(b\right)$
(because there is an extension of $M$ to a model of $T_{S}^{\forall}$
where such a $c$ exists). So by compactness there is some $n$ such
that for every model $M\models T$ and every $a<_{0}^{M}b\in\Suc\left(P_{0}^{M}\right)$,
if $b$ is greater than $n$ successors of $a$, then there is some
$c$ with $a<_{0}^{M}c<_{0}^{M}b$ and $G_{0,1}^{M}\left(c\right)\neq G_{0,1}^{M}\left(b\right)$.
But there is a model $M'$ of $T_{S}^{\forall}$ with some $a<_{0}^{M'}b$
such that $b$ is the $\left(n+1\right)$'th successor of $a$ and
$G_{0,1}^{M'}$ is constant on the interval $\left(a,b\right]$. Since
every model of $T_{S}^{\forall}$ can be extended to a model of $T$
this is a contradiction. 
\end{rem}

\subsection*{Model completion}

Here we will prove the existence of the model completion $T_{S}$
of $T_{S}^{\forall}$.
\begin{notation}
\label{nota:substructures of trees}If $S_{1},S_{2}$ are standard
trees, we shall treat them as structures in the language $\left\{ <_{\suc},<\right\} $,
so when we write $S_{1}\subseteq S_{2}$, we mean that $S_{1}$ is
a substructure of $S_{2}$ in this language (which means that if $b$
is the successor of $a$ in $S_{1}$, it remains such in $S_{2}$).

When $M$ is a model of $T_{S}$, we may write $<$, $\suc$, etc.
instead of $\suc_{\eta}$, $<_{\eta}$ etc. or $\suc_{\eta}^{M}$,
$<_{\eta}^{M}$ etc. where $M$ and $\eta$ are clear from the context. \end{notation}
\begin{rem}
\label{rem:TUnivJEP}$ $ Let $S$ be a standard tree. The following
is not hard to see:
\begin{enumerate}
\item $T_{S}^{\forall}$ is a universal theory.
\item $T_{S}^{\forall}$ has the joint embedding property (JEP).
\item If $S_{1}\subseteq S_{2}$ then $T_{S_{1}}^{\forall}\subseteq T_{S_{2}}^{\forall}$
and moreover, if $M\models T_{S_{2}}^{\forall}$ is existentially
closed, $M\upharpoonright L_{S_{1}}$ is an existentially closed model
of $T_{S_{1}}^{\forall}$.
\end{enumerate}
\end{rem}
We will need some technical closure operators. 
\begin{defn}
$ $Assume $S$ is a \uline{finite} standard tree.
\begin{enumerate}
\item Suppose $\Sigma$ is a finite set of \uline{terms} from $L_{S}$.
We define the following closure operators on terms:

\begin{enumerate}
\item $\cl_{\wedge}^{S}\left(\Sigma\right)=\Sigma\cup\bigcup\left\{ \wedge_{\eta}\left(\Sigma^{2}\right)\left|\,\eta\in S\right.\right\} =\Sigma\cup\left\{ t_{1}\wedge_{\eta}t_{2}\left|\, t_{1},t_{2}\in\Sigma,\eta\in S\right.\right\} $.
\item $\cl_{G}^{S}\left(\Sigma\right)=\Sigma\cup\bigcup\left\{ G_{\eta_{1},\eta_{2}}\left(\Sigma\right)\left|\,\eta_{1}<_{\suc}\eta_{2}\in S\right.\right\} $.
\item $\cl_{\limb}^{S}\left(\Sigma\right)=\Sigma\cup\bigcup\left\{ \minb_{\eta}\left(\Sigma\right)\left|\,\eta\in S\right.\right\} $.
\item $\cl^{0,S}\left(\Sigma\right)=\cl_{G}^{S}\left(\cl_{\limb}^{S}\left(\cl_{\wedge}^{S}\left(\cdots\left(\cl_{G}^{S}\left(\cl_{\limb}^{S}\left(\cl_{\wedge}^{S}\left(\Sigma\right)\right)\right)\right)\right)\right)\right)$
where the number of compositions is the length of the longest branch
in $S$.
\item $\cl_{\suc}^{S}\left(\Sigma\right)=\bigcup\left\{ \suc_{\eta}\left(\Sigma^{2}\right)\cup\pre_{\eta}\left(\Sigma\right)\left|\,\eta\in S\right.\right\} \cup\Sigma$.
\item $\cl^{S}\left(\Sigma\right)=\cl^{0,S}\left(\cl_{\suc}^{S}\left(\Sigma\right)\right)$.
\end{enumerate}
\item Denote $\cl^{\left(0\right),S}=\cl^{0,S}$ and for a number $0<k<\omega$,
$\cl^{\left(k\right),S}\left(\Sigma\right)=\cl^{S}\left(\cl^{\left(k-1\right),S}\left(S\right)\right)$.
\item If $\bar{t}=\left\langle t_{i}\left|\, i<n\right.\right\rangle $
is an $n$-tuple of terms then $\cl^{S}\left(\bar{t}\right)$ is $\cl^{S}\left(\left\{ t_{i}\left|\, i<n\right.\right\} \right)$,
and similarly define the other closure operators for tuples of terms. 
\item For a model $M\models T_{S}^{\forall}$, and $\bar{a}\in M^{<\omega}$,
define $\cl^{0,S}\left(\bar{a}\right)=\left(\cl^{0,S}\left(\bar{x}\right)\right)^{M}\left(\bar{a}\right)$
where $\bar{x}$ is a sequence of variables in the length of $\bar{a}$.
Similarly define $\cl_{\wedge}^{S}\left(\bar{a}\right),\cl_{\limb}^{S}\left(\bar{a}\right),\cl_{G}^{S}\left(\bar{a}\right),\cl_{\suc}^{S}\left(\bar{a}\right)$
and $\cl^{\left(k\right),S}\left(\bar{a}\right)$. For a set $A\subseteq M$,
define $\cl^{0,S}\left(A\right)=\cl^{0,S}\left(\bar{a}\right)$ where
$\bar{a}$ is an enumeration of $A$, and similarly for the other
closure operators.
\end{enumerate}
\end{defn}
We will usually omit the superscript $S$ when it is clear from the
context. 
\begin{claim}
\label{cla:cl^0 is closed under lim, G, wedge}Assume $S$ is a finite
standard tree. For $A\subseteq M\models T_{S}^{\forall}$, $\cl^{0}\left(A\right)$
is closed under $\wedge_{\eta}$, $\limb_{\eta}$ and $G_{\eta_{1},\eta_{2}}$
for all $\eta$ and $\eta_{1}<_{\suc}\eta_{2}$ in $S$. So it is
the substructure generated by $A$ in the language $L_{S}'$ (recall
that $L_{S}'=L_{S}\backslash\set{\pre_{\eta},\suc_{\eta}}{\eta\in S}$).\end{claim}
\begin{proof}
[Proof (sketch).]Note that $\cl_{\limb}\left(\cl_{\wedge}\left(A\right)\right)$
is closed under $\limb_{\eta}$ and $\wedge_{\eta}$ for all $\eta\in S$. 

For $n<\omega$, let $\cl^{0,\left(n\right)}\left(A\right)=\cl_{G}\left(\cl_{\limb}\left(\cl_{\wedge}\left(\cdots\left(\cl_{G}\left(\cl_{\limb}\left(\cl_{\wedge}\left(A\right)\right)\right)\right)\right)\right)\right)$
where there are $n$ compositions. For $\eta\in S$, let $r\left(\eta\right)=\left|\left\{ \nu\in S\left|\,\nu\leq\eta\right.\right\} \right|$,
so $\cl^{0}=\cl^{0,\left(\max\left\{ r\left(\eta\right)\left|\,\eta\in S\right.\right\} \right)}$. 

Let $B\supseteq A$ be the closure of $A$ in $M$ under $\wedge_{\eta}$,
$\limb_{\eta}$ and $G_{\eta_{1},\eta_{2}}$ for all $\eta$ and $\eta_{1}<_{\suc}\eta_{2}$
in $S$. Then $B$ is in fact $\cl^{0,\left(\omega\right)}\left(A\right)=\bigcup\left\{ \cl^{0,\left(n\right)}\left(A\right)\left|\, n<\omega\right.\right\} $.
Now, by induction on $r\left(\eta\right)$ it is easy to see that
$B\cap P_{\eta}=\cl^{0,\left(r\left(\eta\right)\right)}\left(A\right)\cap P_{\eta}$.
Hence $B=\cl^{0}\left(A\right)$.\end{proof}
\begin{claim}
\label{cla:PolBoundOnCl} Assume $S$ is a finite standard tree. For
every $k<\omega$, there is a polynomial $f_{k}^{S}$ such that for
every finite subset $A$ of a model $M$ of $T_{S}^{\forall}$, $\left|\cl^{\left(k\right)}\left(A\right)\right|\leq f_{k}^{S}\left(\left|A\right|\right)$.
Moreover, we can choose $f_{k}^{S}$ so that it is linear (i.e., of
degree $1$).\end{claim}
\begin{proof}
The fact that $f_{k}^{S}$ exists is trivial. For the moreover part,
letting $U=\left\{ \wedge,G,\lim,\suc\right\} $, it is enough to
show that there are $\set{d_{\square}\in\Nn}{\square\in U}$ such
that for every finite $A$, $\square\in U$, $\left|\cl_{\square}\left(A\right)\right|\leq d_{\square}\cdot\left|A\right|$. 

We can choose $d_{\lim}=2$ and $d_{G}=2^{\left|S\right|^{2}}$.

For $\square=\wedge$, note that for all $a\in M$, $\cl_{\wedge}\left(A\cup\left\{ a\right\} \right)=\cl_{\wedge}\left(A\right)\cup\left\{ a,\max\set{a\wedge_{\eta}b}{b\in A}\right\} $
where $a\in P_{\eta}$ (this follows from the fact that if $a\wedge b<b\wedge b'$
then $a\wedge b'=a\wedge b$). So by induction on $\left|A\right|$,
$\left|\cl_{\wedge}\left(A\right)\right|\leq2\left|A\right|$.

For $\square=\suc$, note that for $a\in M$ such that for no $b\in A$,
$b\geq a$, $\cl_{\suc}\left(A\cup\left\{ a\right\} \right)\subseteq\cl_{\suc}\left(A\right)\cup\left\{ a,\pre_{\eta}\left(a\right),\suc_{\eta}\left(a',a\right)\right\} $
where $a\in P_{\eta}$ and $a'=\max\set{b\in A}{b<_{\eta}a}$ (it
may be that this set is empty or that $a$ is a limit element, so
the closure may be smaller). Hence by induction on $\left|A\right|$,
$\left|\cl_{\suc}\left(A\right)\right|\leq3\left|A\right|$. \end{proof}
\begin{rem}
Note that although the degree of $f_{k}^{S}$ in Claim \ref{cla:PolBoundOnCl}
is $1$, the coefficients do depend on $k$ and $S$. \end{rem}
\begin{defn}
$ $\label{def:suc-rank}Assume $S$ is a finite standard tree. 
\begin{enumerate}
\item For a term $t$ of $L_{S}$, we define its \emph{successor rank} as
follows: if $\suc$ and $\pre$ do not appear in $t$, then $r_{\suc}\left(t\right)=0$.
For two terms $t_{1},t_{2}$: $r_{\suc}\left(\suc_{\eta}\left(t_{1},t_{2}\right)\right)=\max\left\{ r_{\suc}\left(t_{1}\right),r_{\suc}\left(t_{2}\right)\right\} +1$,
$r_{\suc}\left(\pre_{\eta}\left(t_{1}\right)\right)=r_{\suc}\left(t_{1}\right)+1$,
$r_{\suc}\left(t_{1}\wedge t_{2}\right)=\max\left\{ r_{\suc}\left(t_{1}\right),r_{\suc}\left(t_{2}\right)\right\} $,
$r_{\suc}\left(G_{\eta_{1},\eta_{2}}\left(t_{1}\right)\right)=r_{\suc}\left(t_{1}\right)$
and $r_{\suc}\left(\minb_{\eta}\left(t_{1}\right)\right)=r_{\suc}\left(t_{1}\right)$.
\item For a quantifier free formula $\varphi$ in $L_{S}$, let $r_{\suc}\left(\varphi\right)$
be the maximal rank of a term appearing in $\varphi$.
\item For $k<\omega$ and an $n$-tuple of variables $\bar{x}$, denote
by $\Delta_{k}^{\bar{x},S}$ the set of all atomic formulas $\varphi\left(\bar{x}\right)$
in $L_{S}$ such that for every term $t$ in $\varphi$, $t\in\cl^{\left(k\right)}\left(\bar{x}\right)$.
Note that since $\cl^{\left(k\right)}\left(\bar{x}\right)$ is a finite
set, so is $\Delta_{k}^{\bar{x},S}$. 
\end{enumerate}
\end{defn}
\begin{claim}
\label{cla:cl^kRank}Suppose $S$ is a finite standard tree. Assume
that $M\models T_{S}^{\forall}$, $n<\omega$, $\bar{a}\in M^{n}$
and $\bar{x}$ a tuple of $n$ variables. Then $\cl^{\left(k\right)}\left(\bar{a}\right)=\left\{ t^{M}\left(\bar{a}\right)\left|\, r_{\suc}\left(t\left(\bar{x}\right)\right)\leq k\right.\right\} $.\end{claim}
\begin{proof}
The inclusion $\subseteq$ is clear. The other direction follows by
induction on $k$ and $t$.

For instance, suppose $r_{\suc}\left(t\left(\bar{x}\right)\right)=k$
and $t=G_{\eta_{1},\eta_{2}}\left(t_{1}\right)$, then by induction
there is some $t_{2}\in\cl^{\left(k\right)}\left(\bar{x}\right)$
such that $t_{1}^{M}\left(\bar{a}\right)=t_{2}^{M}\left(\bar{a}\right)$.
If $t_{2}\left(\bar{a}\right)\notin\Suc\left(P_{\eta_{1}}^{M}\right)$,
then $t_{2}^{M}\left(\bar{a}\right)$ is not in the domain of $G_{\eta_{1},\eta_{2}}^{M}$
and so $t^{M}\left(\bar{a}\right)=t_{2}^{M}\left(\bar{a}\right)$.
If $t_{2}^{M}\left(\bar{a}\right)\in\Suc\left(P_{\eta_{1}}^{M}\right)$,
then by the proof of Claim \ref{cla:cl^0 is closed under lim, G, wedge},
there is some $t_{3}\left(\bar{x}\right)\in\cl^{0,\left(r\left(\eta_{1}\right)\right)}\left(\cl_{\suc}\left(\cl^{\left(k-1\right)}\left(\bar{x}\right)\right)\right)$
such that $t_{2}^{M}\left(\bar{a}\right)=t_{3}^{M}\left(\bar{a}\right)$
(if $k=0$, then $t_{3}\left(\bar{x}\right)\in\cl^{0,\left(r\left(\eta_{1}\right)\right)}\left(\bar{x}\right)$).
So $t_{4}=G_{\eta_{1},\eta_{2}}\left(t_{3}\right)\in\cl^{\left(k\right)}\left(\bar{x}\right)$
and $t^{M}\left(\bar{a}\right)=t_{4}^{M}\left(\bar{a}\right)$. If
$t=s_{1}\wedge_{\eta}s_{2}$, then by induction there are $s_{3},s_{4}\in\cl^{\left(k\right)}\left(\bar{x}\right)$
such that $t^{M}\left(\bar{a}\right)=s_{3}^{M}\left(\bar{a}\right)\wedge_{\eta}s_{4}^{M}\left(\bar{a}\right)$.
Since $\cl^{\left(k\right)}\left(\bar{a}\right)$ is closed under
$\wedge$ (by Claim \ref{cla:cl^0 is closed under lim, G, wedge}),
there is some $s_{5}\in\cl^{\left(k\right)}\left(\bar{x}\right)$
such that $t^{M}\left(\bar{a}\right)=s_{5}^{M}\left(\bar{a}\right)$. \end{proof}
\begin{defn}
\label{def:k-isomorphism}Suppose $S$ is a finite standard tree and
$k<\omega$. Let $M_{1},M_{2}\models T_{S}^{\forall}$. 
\begin{enumerate}
\item Suppose $n<\omega$ and $\bar{a}\in M_{1}^{n}$, $\bar{b}\in M_{2}^{n}$.
We say that $\bar{a}\equiv_{k}^{S}\bar{b}$ if there is an isomorphism
of $L_{S}'$-structures from $\cl^{\left(k\right)}\left(\bar{a}\right)$
to $\cl^{\left(k\right)}\left(\bar{b}\right)$ taking $\bar{a}$ to
$\bar{b}$ (recall that $L_{S}'=L_{S}\backslash\set{\pre_{\eta},\suc_{\eta}}{\eta\in S}$).
In this notation we assume that $M_{1},M_{2}$ are clear from the
context. 
\item If $A\subseteq M_{1}$, $B\subseteq M_{2}$ are two finite subsets
of $M_{1}$ and $M_{2}$, we write $A\xrightarrow[k]{S,f}B$ when
$f$ extends some $L_{S}'$-isomorphism $f':\cl^{\left(k\right)}\left(A\right)\to\cl^{\left(k\right)}\left(B\right)$
such that $f'\left(A\right)=B$. So this is equivalent to saying that
$B=\set{f\left(a\right)}{a\in A}$, $\sequence a{a\in A}\equiv_{k}^{S}\sequence{f\left(a\right)}{a\in A}$
and $f\upharpoonright\cl^{\left(k\right)}\left(A\right)$ witnesses
this. 
\end{enumerate}
\end{defn}
Recall (from the notation section in the beginning of Section \ref{sec:Preliminaries}),
that for a finite set of formulas $\Delta$, by writing $\Delta\left(\bar{x};\bar{y}\right)$
we mean that we assign to it a partition of the free variables appearing
in it. In that case, for $\bar{b}$ of the same length as $\bar{x}$,
$\tp_{\Delta\left(\bar{x};\bar{y}\right)}\left(\bar{b}/A\right)$
is the set of formulas $\varphi\left(\bar{x},\bar{a}\right)$ such
that $\varphi\left(\bar{x},\bar{y}\right)\in\Delta$, $\bar{a}\in A^{\lg\left(\bar{y}\right)}$
and $\C\models\varphi\left(\bar{b},\bar{a}\right)$. If the partition
$\left(\bar{x};\bar{y}\right)$ is clear, then we omit it from the
notation. 

Recall also that $\Delta_{k}^{\bar{x},S}$ is the set of all atomic
formulas $\varphi\left(\bar{x}\right)$ in $L_{S}$ such that for
every term $t$ in $\varphi$, $t\in\cl^{\left(k\right)}\left(\bar{x}\right)$.
\begin{defn}
\label{def:k-type}Suppose $S$ is a finite standard tree. For $M\models T_{S}^{\forall}$,
$\bar{a}\in M^{<\omega}$, $A\subseteq M$ a finite set, and $k<\omega$,
let $\tp_{k}^{S}\left(\bar{a}/A\right)=\tp_{\Delta_{k}^{\bar{x}\bar{y},S}}\left(\bar{a}/A\right)$
where $\lg\left(\bar{x}\right)=\lg\left(\bar{a}\right)$ and  $\bar{y}$
is of length $\left|A\right|$. This is the \emph{$k$-type of $\bar{a}$
over $A$}. 
\end{defn}
In Definitions \ref{def:suc-rank}, \ref{def:k-isomorphism} and \ref{def:k-type},
we omit $S$ from the superscript when it is clear from the context. 
\begin{claim}
\label{cla:quFreeTpEquiv} Suppose $S$ is a finite standard tree.
Assume $M_{1},M_{2}\models T_{S}^{\forall}$. Assume that $\bar{a}\in M_{1}^{n},\bar{b}\in M_{2}^{n}$
for some $n<\omega$, $\bar{x}$ a tuple of $n$ variables and assume
$k<\omega$. Then the following are equivalent:
\begin{enumerate}
\item $\bar{a}\equiv_{k}\bar{b}$.
\item $\tp_{k}\left(\bar{a}\right)=\tp_{k}\left(\bar{b}\right)$.
\item For every quantifier free formula $\varphi\left(\bar{x}\right)$ in
$L_{S}$ with $r_{\suc}\left(\varphi\right)\leq k$, $M_{1}\models\varphi\left(\bar{a}\right)\Leftrightarrow M_{2}\models\varphi\left(\bar{b}\right)$.
\item The tuples $\sequence{t\left(\bar{a}\right)}{t\in\cl^{\left(k\right)}\left(\bar{x}\right)}$
and $\sequence{t\left(\bar{b}\right)}{t\in\cl^{\left(k\right)}\left(\bar{x}\right)}$
have the same quantifier free type in $L_{S}'$. 
\end{enumerate}
\end{claim}
\begin{proof}
(1) implies (2): assume $\bar{a}\equiv_{k}\bar{b}$ and $f:\cl^{\left(k\right)}\left(\bar{a}\right)\to\cl^{\left(k\right)}\left(\bar{b}\right)$
is an $L_{S}'$-isomorphism taking $\bar{a}$ to $\bar{b}$. It is
easy to see by induction on $t$ and $k$ that for every term $t\in\cl^{\left(k\right)}\left(\bar{x}\right)$,
$f\left(t\left(\bar{a}\right)\right)=t\left(\bar{b}\right)$, and
so $\tp_{k}\left(\bar{a}\right)=\tp_{k}\left(\bar{b}\right)$.

(2) implies (3): this follows from Claim \ref{cla:cl^kRank} --- for
every term $t\left(\bar{x}\right)$ with rank $r_{\suc}\left(t\right)\leq k$
there is a term $t'\in\cl^{\left(k\right)}\left(\bar{x}\right)$ such
that $M_{1}\models t'\left(\bar{a}\right)=t\left(\bar{a}\right)$.
By induction on $k$ and $t$, one can show that since $\tp_{k}\left(\bar{a}\right)=\tp_{k}\left(\bar{b}\right)$,
$M_{2}\models t'\left(\bar{b}\right)=t\left(\bar{b}\right)$ and this
suffices. For instance, suppose $t=s_{1}\wedge_{\eta}s_{2}$. By induction,
there are $s_{3},s_{4}\in\cl^{\left(k\right)}\left(\bar{x}\right)$
such that $s_{1}^{M_{1}}\left(\bar{a}\right)=s_{3}^{M_{1}}\left(\bar{a}\right)$
and $s_{2}^{M_{1}}\left(\bar{a}\right)=s_{4}^{M_{1}}\left(\bar{a}\right)$
and the same equations hold with $M_{2}$ instead of $M_{1}$ and
$\bar{b}$ instead of $\bar{a}$. Since $\cl^{\left(k\right)}\left(\bar{a}\right)$
is closed under $\wedge$, there is some $s_{5}\in\cl^{\left(k\right)}\left(\bar{x}\right)$
such that $M_{1}\models s_{3}\left(\bar{a}\right)\wedge_{\eta}s_{4}\left(\bar{a}\right)=s_{5}\left(\bar{a}\right)$,
so 
\[
s_{5}^{M_{1}}\left(\bar{a}\right)=\max\set{s^{M_{1}}\left(\bar{a}\right)}{s\in\cl^{\left(k\right)}\left(\bar{x}\right),M_{1}\models s\left(\bar{a}\right)\leq_{\eta}s_{3}\left(\bar{a}\right),s_{4}\left(\bar{a}\right)}.
\]
By (2), the same equation holds if we replace $M_{1}$ with $M_{2}$
and $\bar{a}$ with $\bar{b}$. Since $\cl^{\left(k\right)}\left(\bar{b}\right)$
is closed under $\wedge$, it follows that $M_{2}\models s_{3}\left(\bar{b}\right)\wedge_{\eta}s_{4}\left(\bar{b}\right)=s_{5}\left(\bar{b}\right)$.

(3) implies (4): since formulas in $L_{S}'$ do not increase the successor
rank, this is clear.  

(4) implies (1): the map taking $t\left(\bar{a}\right)$ to $t\left(\bar{b}\right)$
for every term $t\in\cl^{\left(k\right)}\left(\bar{x}\right)$ is
a well defined isomorphism of $L_{S}'$ structures.
\end{proof}
Similarly, we have:
\begin{claim}
\label{cla:reduc}Suppose $S$ is a finite standard tree. Let $M\models T_{S}^{\forall}$,
$n<\omega$, $\bar{a},\bar{b}\in M^{n}$, $\bar{x}$ a tuple of $n$
variables and $k,k_{1},k_{2}<\omega$. 
\begin{enumerate}
\item if $\bar{a}\equiv_{k}\bar{b}$ then there is a unique isomorphism
that shows it. Namely, for each $t\in\cl^{\left(k\right)}\left(\bar{x}\right)$,
the isomorphism $f$ must satisfy $f\left(t\left(\bar{a}\right)\right)=t\left(\bar{b}\right)$.
\item Assume $k_{2}\geq k_{1}$. Then $\bar{a}\equiv_{k_{2}}\bar{b}$ implies
$\bar{a}\equiv_{k_{1}}\bar{b}$.
\item If $\bar{a}\bar{a}'\equiv_{k}\bar{b}\bar{b}'$ then $\bar{a}\equiv_{k}\bar{b}$.
\item If $\bar{a}\equiv_{k+1}\bar{b}$, witnessed by $f$, then $\cl\left(\bar{a}\right)\xrightarrow[k]{S,f}\cl\left(\bar{b}\right)$. 
\item If $S'\subseteq S$, and $\bar{a}\equiv_{k}^{S}\bar{b}$ then $\bar{a}\equiv_{k}^{S'}\bar{b}$
(when $\bar{a}$ and $\bar{b}$ are considered as tuples in $M_{1}\upharpoonright L_{S'}$
and $M_{2}\upharpoonright L_{S'}$). 
\end{enumerate}
\end{claim}
Before proceeding to prove the main quantifier elimination lemma,
let us give two more important definitions:
\begin{defn}
Suppose $S$ is a standard tree. Suppose $M\models T_{S}^{\forall}$,
$\eta\in S$ and $a,b\in P_{\eta}^{M}$. We say that the \emph{distance}
between $a$ and $b$ is $n$ if $a<_{\eta}b$ and $b$ is the $n$-th
successor of $a$ or vice-versa. We say the distance is infinite if
for no $n<\omega$ the distance is $n$. Denote this by $d\left(a,b\right)=n$.
\end{defn}
For a set $A\subseteq M\models T_{S}^{\forall}$, we denote by $\Suc\left(A\right)$
the set of all successors in $A$. 
\begin{defn}
\label{def:eq-relation}Suppose $S$ is a standard tree, $\eta\in S$,
$M\models T_{S}^{\forall}$ and $A\subseteq M$. Let $R_{\eta}^{A}\subseteq\Suc\left(A\right)^{2}$
be the following relation: $\left(x,y\right)\in R_{\eta}^{A}$ iff
$\lim\left(x\right)=\lim\left(y\right)$ and $x$ and $y$ are comparable
($x<_{\eta}y$ or $y\leq_{\eta}x$). Let $\sim_{\eta}^{A}$ be the
the transitive closure of $R_{\eta}^{A}$ (so it is an equivalence
relation on $\Suc\left(A\right)$). 
\end{defn}
So the equivalence relation $\sim_{\eta}$ determines the function
$G_{\eta,\eta'}$ for $\eta<_{\suc}\eta'$ from $S$: if $a,b\in P_{\eta}^{M}$
for $M\models T_{S}^{\forall}$ and $a\sim_{\eta}^{M}b$ then $G_{\eta,\eta'}\left(a\right)=G_{\eta,\eta'}\left(b\right)$. 
\begin{lem}
\label{lem:quElLemma} (Quantifier elimination lemma) For every finite
standard tree $S$, and $m_{1},n,k<\omega$, there is $m_{2}=m_{2}\left(m_{1},k,S\right)<\omega$
such that if:
\begin{itemize}
\item $M_{1},M_{2}\models T_{S}^{\forall}$ are existentially closed.
\item $\bar{a}\in M_{1}^{n}$ and $\bar{b}\in M_{2}^{n}$.
\item $\bar{a}\equiv_{m_{2}}\bar{b}$.
\end{itemize}
Then for all $\bar{c}\in M_{1}^{k}$ there is some $\bar{d}\in M_{2}^{k}$
such that $\bar{c}\bar{a}\equiv_{m_{1}}\bar{d}\bar{b}$.
\end{lem}
(Note that $m_{2}$ does not depend on $n$.)
\begin{proof}
The proof is by induction on $\left|S\right|$. Given $S$, we will
show that the lemma holds for all $m_{1}$ and $k$. Without loss
of generality $k=1$: by induction one can choose $m_{2}\left(m_{1},k+1,S\right)=m_{2}\left(m_{2}\left(m_{1},k,S\right),1,S\right)$.
We may also assume that $m_{1}>0$. 

We may assume that $m_{2}\left(m_{1},k,S'\right)>\max\left\{ m_{1},k,\left|S'\right|\right\} $
for all $S'\subsetneq S$ (by enlarging $m_{2}$ if necessary).

For $\left|S\right|=0$ the claim is trivial because $T_{S}^{\forall}$
is just the theory of a set with no structure.

Assume $0<\left|S\right|$. Let $\eta_{0}$ be the root of $S$, $S_{0}=\left\{ \eta_{0}\right\} $
and partition $S$ as $S=\bigcup\left\{ S_{i}\left|\, i<m\right.\right\} $
where for $i\geq1$, the $S_{i}$'s are the connected components of
$S$ above $\eta_{0}$ (note that $S_{i}\subseteq S$, see Notation
\ref{nota:substructures of trees}). Let 
\[
m_{2}=m_{2}\left(m_{1},1,S\right)=\max\left\{ 2m_{2}\left(m_{1},K,S_{i}\right)\left|\,1\leq i<m\right.\right\} +2m_{1}+1
\]
 where $K=3$. 

Suppose $M_{1},M_{2}$, $\bar{a}$ and $\bar{b}$ are as in the lemma
and let $c\in M_{1}$. 

By assumption there is a unique $L_{S}'$-isomorphism $f:\cl^{\left(m_{2}\right)}\left(\bar{a}\right)\to\cl^{\left(m_{2}\right)}\left(\bar{b}\right)$.

For $i\leq m$, let $P_{S_{i}}=\bigvee\left\{ P_{\eta}\left|\,\eta\in S_{i}\right.\right\} $,
$A_{i}=\cl^{\left(m_{1}\right)}\left(\bar{a}\right)\cap P_{S_{i}}^{M_{1}}$
and $B_{i}=\cl^{\left(m_{1}\right)}\left(\bar{b}\right)\cap P_{S_{i}}^{M_{2}}$.

Since $\bar{a}\equiv_{m_{2}}\bar{b}$, it follows that $\cl^{\left(m_{2}\left(m_{1},K,S_{i}\right)\right)}\left(\bar{a}\right)\xrightarrow[m_{2}\left(m_{1},K,S_{i}\right)]{f}\cl^{\left(m_{2}\left(m_{1},K,S_{i}\right)\right)}\left(\bar{b}\right)$
and in particular $A_{i}\xrightarrow[m_{2}\left(m_{1},K,S_{i}\right)]{S_{i},f}B_{i}$
(see Claim \ref{cla:reduc} (4) and (5)).

We divide into cases:
\begin{casenv}
\item $c\notin P_{\eta}^{M_{1}}$ for every $\eta\in S$.

Here finding $d$ is easy due to the fact that $M_{1}$ and $M_{2}$
are existentially closed.

\item $c\in P_{S_{i}}^{M_{1}}$ for some $1\leq i\leq m$.

$A_{i}\xrightarrow[m_{2}\left(m_{1},K,S_{i}\right)]{S_{i},f}B_{i}$
(as subsets of $M_{1}\upharpoonright L_{S_{i}}$ and $M_{2}\upharpoonright L_{S_{i}}$),
so by the induction hypothesis (and by Remark \ref{rem:TUnivJEP}
(3)) we can find $d\in M_{2}$ and extend $f\upharpoonright\cl^{\left(m_{1}\right)}\left(A_{i}\right)$
to an $L_{S_{i}}'$-isomorphism $f':\cl^{\left(m_{1}\right)}\left(\left\{ c\right\} \cup A_{i}\right)\to\cl^{\left(m_{1}\right)}\left(\left\{ d\right\} \cup B_{i}\right)$
taking $c$ to $d$. Note that $f'$ is also an $L_{S}'$-isomorphism.
It follows that
\[
f\upharpoonright\cl^{\left(m_{1}\right)}\left(\bar{a}\right)\cup f'\upharpoonright\cl^{\left(m_{1}\right)}\left(c\bar{a}\right)
\]
 is an $L_{S}'$-isomorphism from $\cl^{\left(m_{1}\right)}\left(c\bar{a}\right)$
to $\cl^{\left(m_{1}\right)}\left(d\bar{b}\right)$ that shows that
$c\bar{a}\equiv_{m_{1}}d\bar{b}$ (note that $P_{S_{j}}^{M_{1}}\cap\cl^{\left(m_{1}\right)}\left(\bar{a}c\right)=A_{j}$
for $j\neq i$ and that if $x\in\cl^{\left(m_{1}\right)}\left(\bar{a}c\right)\cap P_{S_{i}}^{M_{1}}$
then $x\in\cl^{\left(m_{1}\right)}\left(\left\{ c\right\} \cup A_{i}\right)$,
and so the domain is indeed $\cl^{\left(m_{1}\right)}\left(c\bar{a}\right)$).

\item $c\in P_{\eta_{0}}$.

For notational simplicity, let $<$ be $<_{\eta_{0}}$, $\limb$ be
$\limb_{\eta_{0}}$, $\sim$ be $\sim_{\eta_{0}}$ and $\wedge$ be
$\wedge_{\eta_{0}}$.

Let $A_{0}'=\cl^{\left(0\right)}\left(\bar{a}\right)\cap P_{\eta_{0}}^{M_{1}}$
(so this is the closure of $\bar{a}$ inside $P_{\eta_{0}}$ under
$\wedge$ and $\lim$), $B_{0}'=\cl^{\left(0\right)}\left(\bar{b}\right)\cap P_{\eta_{0}}^{M_{2}}$,
$F=\cl^{\left(m_{1}\right)}\left(A_{0}'\cup\left\{ c\right\} \right)\cap P_{\eta_{0}}^{M_{1}}$
and $\eta_{i}=\min\left(S_{i}\right)$ for $1\leq i\leq m$.

Note that $F$ is really just $\cl_{\suc}^{\left(m_{1}\right)}\left(\cl^{\left(0\right)}\left(A_{0}'\cup\left\{ c\right\} \right)\right)$.

Say that an element of $F$ is \emph{new} if it is a successor and
is not $\sim^{F}$-equivalent to any element from $A_{0}$ (note:
$A_{0}$ and \uline{not} $A_{0}'$). We will prove the following
claim:
\begin{claimi*}

\begin{enumerate}
\item There are at most $K$ many $\sim^{F}$-equivalence classes of new
elements in $F$. For each one choose a representative. Enumerate
them as $\sequence{c_{l}}{l<K'}$ for $K'\leq K$. 
\item There is a model $M_{3}'$ of $T_{S_{0}}^{\forall}$, an $L_{S_{0}}'$-isomorphism
$f'$ and $d'\in M_{3}'$ such that $M_{3}'\supseteq P_{\eta_{0}}^{M_{2}}$,
$f'\upharpoonright A_{0}=f$, $A_{0}'\cup\left\{ c\right\} \xrightarrow[m_{1}]{S_{0},f'}B_{0}'\cup\left\{ d'\right\} $
and $f'\left(c\right)=d'$ (so the domain of $f'$ is $F$). 
\item Moreover, for $l<K'$, $f'\left(c_{l}\right)$ are pairwise non-$\sim^{M_{3}}$-equivalent
and they are not $\sim^{M_{3}}$-equivalent to any element from $\Suc\left(P_{\eta_{0}}^{M_{2}}\right)$. 
\end{enumerate}
\end{claimi*}

Suppose first that Claim I holds.

For $1\leq i\leq m$ let $c_{l}^{i}=G_{\eta_{0},\eta_{i}}\left(c_{l}\right)$.

Fix $1\leq i\leq m$. By assumption, $A_{i}\xrightarrow[m_{2}\left(m_{1},K,S_{i}\right)]{S_{i},f}B_{i}$,
so by the induction hypothesis there are $d_{l}^{i}\in M_{2}$ for
$l<K'$ and an $L_{S_{i}}'$-isomorphism $g_{i}$ extending $f\upharpoonright\cl^{\left(m_{1}\right)}\left(A_{i}\right)$
such that $g_{i}\left(c_{l}^{i}\right)=d_{l}^{i}$ and $A_{i}\cup\set{c_{l}^{i}}{l<K'}\xrightarrow[m_{1}]{S_{i},g_{i}}B_{i}\cup\set{d_{l}^{i}}{l<K'}$. 
\begin{claimii*}
There exists a model $M_{3}\models T_{S}^{\forall}$ satisfying $P_{\eta_{0}}^{M_{3}}=P_{\eta_{0}}^{M_{3}'}$,
$M_{3}\supseteq M_{2}$ and $G_{\eta_{0},\eta_{i}}^{M_{3}}\left(f'\left(c_{l}\right)\right)=d_{l}^{i}$
for $l<K'$ and $1\leq i\leq m$. 
\begin{proof}
(of Claim II) Since $M_{3}'\models T_{S_{0}}^{\forall}$, $M_{2}\models T_{S}^{\forall}$
and $P_{\eta_{0}}^{M_{3}'}\supseteq P_{\eta_{0}}^{M_{2}}$ the only
thing we must show is that $G_{\eta_{0},\eta_{i}}$ defined in Claim
II is well defined and can be extended to a regressive function. This
follows directly from Claim I (3). 
\end{proof}
\end{claimii*}

Define 
\[
g=f\upharpoonright\cl^{\left(m_{1}\right)}\left(\bar{a}\right)\cup f'\upharpoonright\cl^{\left(m_{1}\right)}\left(\bar{a}c\right)\cup\bigcup\left\{ g_{i}\upharpoonright\cl^{\left(m_{1}\right)}\left(\bar{a}c\right)\left|\,1\leq i<m\right.\right\} .
\]
We claim that $g$ is an $L_{S}'$-isomorphism extending $f\upharpoonright\cl^{\left(m_{1}\right)}\left(\bar{a}\right)$
from $\cl^{\left(m_{1}\right)}\left(\bar{a}c\right)$ to $\cl^{\left(m_{1}\right)}\left(\bar{a}d'\right)$
sending $c$ to $d$. It is easy to see that $g$ is well defined
as a function. To see that it is an $L_{S}'$-isomorphism we only
need to show that if $e\in\cl^{\left(m_{1}\right)}\left(\bar{a}c\right)$
is a successor and $1\leq i\leq m$ then $G_{\eta_{0},\eta_{i}}^{M_{3}}\left(f'\left(e\right)\right)=g_{i}\left(G_{\eta_{0},\eta_{i}}^{M_{1}}\left(e\right)\right)$.
Suppose $e\sim^{F}b$ where $b\in A_{0}$, then $f'\left(e\right)\sim^{M_{3}}f'\left(b\right)$,
$G_{\eta_{0},\eta_{i}}^{M_{1}}\left(e\right)=G_{\eta_{0},\eta_{i}}^{M_{1}}\left(b\right)$
and $G_{\eta_{0},\eta_{i}}^{M_{3}}\left(f'\left(e\right)\right)=G_{\eta_{0},\eta_{i}}^{M_{3}}\left(f'\left(b\right)\right)$.
Now we are done since: 
\[
G_{\eta_{0},\eta_{i}}^{M_{3}}\left(f'\left(b\right)\right)=G_{\eta_{0},\eta_{i}}^{M_{2}}\left(f\left(b\right)\right)=f\left(G_{\eta_{0},\eta_{i}}^{M_{1}}\left(b\right)\right)=g_{i}\left(G_{\eta_{0},\eta_{i}}^{M_{1}}\left(b\right)\right).
\]
Suppose $e$ is new. Then $e\sim^{F}c_{l}$ for some $l<K'$. But
then $G_{\eta_{0},\eta_{i}}^{M_{1}}\left(e\right)=G_{\eta_{0},\eta_{i}}^{M_{1}}\left(c_{l}\right)=c_{l}^{i}$,
and $g_{i}\left(c_{l}^{i}\right)=d_{l}^{i}$, while $f'\left(e\right)\sim^{M_{3}}f'\left(c_{l}\right)$,
so $G_{\eta_{0},\eta_{i}}^{M_{3}}\left(f'\left(e\right)\right)=G_{\eta_{0},\eta_{i}}^{M_{3}}\left(f'\left(c_{l}\right)\right)=d_{l}^{i}$
by Claim II.

So $c\bar{a}\equiv_{m_{1}}d'\bar{b}$, i.e., $\tp_{m_{1}}\left(c\bar{a}\right)=\tp_{m_{1}}\left(d'\bar{b}\right)$,
and if $\Psi$ is the conjunction of all formulas appearing in $\tp_{m_{1}}\left(c\bar{a}\right)$
then $M_{3}\models\exists x\Psi\left(x\bar{b}\right)$. As $M_{2}$
is existentially closed there is some $d\in M_{2}$ such that $\Psi\left(d\bar{b}\right)$,
i.e., $c\bar{a}\equiv_{m_{1}}d\bar{b}$.

We will be done once we prove Claim I.
\begin{proof}
\renewcommand{\qedsymbol}{}(of Claim I) Again we need to divide into
cases:
\begin{casenv}
\item $c\in A_{0}'$: there is nothing to do. 
\item \label{cas:InBranch}$c$ is in a branch of $A_{0}'$, i.e., there
is $c<y\in A_{0}'$ and assume $y$ is minimal in this sense (it exists
since $A_{0}'$ is closed under $\wedge$). We again divide into cases:

\begin{casenv}
\item There is no $x\in A_{0}'$ below $c$. This means that $c<x$ for
all $x\in A_{0}'$, and even for all $x\in A_{0}$ (since for all
$x\in A_{0}$, there is $x'\in A_{0}'$ such that $\lim\left(x\right)=\lim\left(x'\right)$)
and that $y=\lim\left(y\right)$. There is exactly one $\sim^{F}$-class
of new elements, which is $\left[\suc\left(c,y\right)\right]_{\sim^{F}}$.
In this case (2) and (3) are easy: just let $d'$ be a new element
below $P_{\eta_{0}}^{M_{2}}$ with the same distance from its limit
as $d\left(c,\lim\left(c\right)\right)$ (which can be infinite, and
if $d\left(c,\lim\left(c\right)\right)>2m_{1}$, we can choose $d\left(d',\lim\left(d'\right)\right)=2m_{1}+1$). 
\item There is some $x\in A_{0}'$ such that $x<c$. Assume $x$ is maximal
in this sense.

If $\limb\left(x\right)<\limb\left(y\right)$ then necessarily $\minb\left(x\right)\leq x<c<\minb\left(y\right)=y$.
If $\lim\left(x\right)<\lim\left(c\right)$, then there is one $\sim^{F}$-class
of new elements --- $\left[\suc\left(c,y\right)\right]_{\sim^{F}}$.
Again (2) and (3) are easy: let $\lim\left(d'\right)$ be a new limit
element below $f\left(y\right)$ and above all elements from $M_{2}$
below $f\left(y\right)$ and let $d'$ be with the right distance
from $\lim\left(d'\right)$. If $\lim\left(x\right)=\lim\left(c\right)$,
then there are no new $\sim^{F}$-classes. Moreover, we can choose
$M_{3}'=M_{2}\upharpoonright L_{S_{0}}$ and $d'\in M_{2}$.

If $\minb\left(x\right)=\minb\left(y\right)$ (so also $=\lim\left(c\right)$),
then again there are no new $\sim^{F}$-classes. For (2) and (3),
we must make sure that the distance between $f\left(x\right)$ and
$f\left(y\right)$ is big enough, so that we can place $d'$ in the
right spot between them. In $F\backslash A_{0}$ we may add $m_{1}$
successors to $c$ in the direction of $y$ and $m_{1}$ predecessors.
This is why we chose $m_{2}\geq2m_{1}+1$.

\end{casenv}
\item $c$ starts a new branch in $A_{0}'$, i.e., there is no $y\in A_{0}'$
such that $c<y$. In this case, let $c'=\left\{ \max\left(c\wedge b\right)\left|\, b\in A_{0}'\right.\right\} $.
Note that if there is an element in $\cl_{\wedge}\left(A_{0}'\cup\left\{ c\right\} \right)\backslash A_{0}\cup\left\{ c\right\} $,
it must be $c'$. Adding $c'$ falls under Case \ref{cas:InBranch}
above (if it is indeed new), so the $\sim^{F}$-classes of new elements
will be those which come from $c'$ as before, and perhaps more. Namely,
it can be that $\lim\left(c\right)<c'$ (so $\lim\left(c\right)=\lim\left(c'\right)$)
in which case that is all, or we should add $\left[\suc\left(\lim\left(c\right),c\right)\right]_{\sim^{F}}$
and $\left[\suc\left(c',c\right)\right]_{\sim^{F}}$.

By the previous case, we can first find $M_{3}''\supseteq P_{\eta_{0}}^{M_{2}}$,
an $L_{S_{0}}'$-isomorphism $f''$ and $d''\in M_{3}''$ such that
$f''\upharpoonright A_{0}=f$, $A_{0}'\cup\left\{ c'\right\} \xrightarrow[m_{1}]{S_{0},f'}B_{0}'\cup\left\{ d''\right\} $
and $f''\left(c'\right)=d''$. Then we can just add a new branch starting
at $d''$ to construct $M_{3}'$. 

\end{casenv}
\end{proof}
\end{casenv}
\end{proof}
\begin{claim}
\label{cla:QEAlmost} Let $S$ be a finite standard tree. For every
formula $\varphi\left(\bar{x}\right)$ (with free variables) there
is a quantifier free formula $\psi\left(\bar{x}\right)$ such that
for every existentially closed model $M\models T_{S}^{\forall}$,
we have $M\models\psi\equiv\varphi$.\end{claim}
\begin{proof}
It is enough to check formulas of the form $\exists y\varphi\left(y,\bar{x}\right)$
where $\varphi$ is quantifier free and $\lg\left(\bar{x}\right)=n>0$.
Let $k=r_{\suc}\left(\varphi\right)$. Let $m=m_{2}\left(k,1,S\right)$
from Lemma \ref{lem:quElLemma}. By Claim \ref{cla:quFreeTpEquiv},
if $M_{1},M_{2}\models T_{S}^{\forall}$ are existentially closed
and $\bar{a}\in M_{1},\bar{b}\in M_{2}$ are of length $n$ and $\bar{a}\equiv_{m}\bar{b}$,
then $M_{1}\models\exists y\varphi\left(y,\bar{a}\right)$ iff $M_{2}\models\exists y\varphi\left(y,\bar{b}\right)$.

Assume $\left|\Delta_{m}^{\bar{x}}\right|=N$ and let $\left\{ \varphi_{i}\left|\, i<N\right.\right\} $
be an enumeration of $\Delta_{m}^{\bar{x}}$. For every $\eta:N\to2$,
let $\varphi_{\eta}^{m}\left(\bar{x}\right)=\bigwedge_{i<N}\varphi_{i}^{\eta\left(i\right)}\left(\bar{x}\right)$
(where $\varphi^{0}=\neg\varphi$ and $\varphi^{1}=\varphi$). 

Let 
\[
R=\left\{ \eta:N\to2\left|\,\exists\mbox{ e.c. }M\models T_{S}^{\forall}\,\&\,\exists\bar{c}\in M\left(M\models\varphi_{\eta}^{m}\left(\bar{c}\right)\land\exists y\varphi\left(y,\bar{c}\right)\right)\right.\right\} .
\]
Let $\psi\left(\bar{x}\right)=\bigvee_{\eta\in R}\varphi_{\eta}^{m}\left(\bar{x}\right)$.
By Claim \ref{cla:quFreeTpEquiv} it follows that $\psi$ is the desired
formula.\end{proof}
\begin{cor}
\label{cor:QE}If $M_{1}$ and $M_{2}$ are two existentially closed
models of $T_{S}^{\forall}$ then $M_{1}\equiv M_{2}$ and their theory
eliminates quantifiers.\end{cor}
\begin{proof}
Assume first that $M_{1}\subseteq M_{2}$, then $M_{1}\prec M_{2}$:
for formulas with free variables it follows directly from the previous
claim, and for a sentence $\varphi$ we consider the formula $\varphi\wedge\left(x=x\right)$.

Now the corollary follows from the fact that the theory is universal
(so every model can be extended to an existentially closed one) and
has JEP. \end{proof}
\begin{defn}
Let $S$ be a finite standard tree. Let $T_{S}$ be the theory of
all existentially closed models of $T_{S}^{\forall}$. 
\end{defn}
From Corollary \ref{cor:QE} and the definition of model completion,
we deduce:
\begin{cor}
\label{cor:Model completion}Let $S$ be a finite standard tree. Then
$T_{S}$ is the model completion of $T_{S}^{\forall}$. The theory
$T_{S}$ eliminates quantifiers. Thus $T_{S}^{\forall}$ has AP. 
\end{cor}

\subsection*{NIP}

In this section we will show that $T_{S}$ is dependent. The idea
is to count the number of $\Delta$-types for finite $\Delta$ over
a finite set of parameters $A$, and to show that this number is polynomial
in $\left|A\right|$. Thus, from Fact \ref{fac:DepPolBd} it follows
that $T_{S}$ is dependent. In fact, we will show that we can find
such polynomials $f_{\Delta}$ such that their \uline{degree} does
not depend on $\Delta$, but only on the number of free variables
and on $S$. From this, by Lemma \ref{lem:strong dep. pol bound}
we will conclude that $T_{S}$ is not just dependent but even strongly$^{2}$
dependent.
\begin{defn}
Suppose $S$ is a finite standard tree. Assume $A\subseteq M\models T_{S}$
is a finite set and $k<\omega$.
\begin{enumerate}
\item We say that $a,b\in M$ are $k$-isomorphic over $A$, denoted by
$a\equiv_{A,k}^{S}b$ iff for some (any) enumeration $\bar{a}$ of
$A$, $a\bar{a}\equiv_{k}^{S}b\bar{a}$.
\item Similarly for tuples from $M^{<\omega}$.
\end{enumerate}
\end{defn}
\begin{claim}
\label{cla:TypesOver} Suppose $S$ is a finite standard tree. Assume
$M\models T_{S}^{\forall}$, $k<\omega$, $A\subseteq M$ is finite
and $\bar{a},\bar{b}\in M^{<\omega}$. Then $\bar{a}\equiv_{A,k}\bar{b}$
iff $\tp_{k}\left(\bar{a}/A\right)=\tp_{k}\left(\bar{b}/A\right)$
iff for every quantifier free formula $\varphi\left(\bar{x}\right)$
over $A$ such that $r_{\suc}\left(\varphi\right)\leq k$, $M\models\varphi\left(\bar{a}\right)\leftrightarrow\varphi\left(\bar{b}\right)$. \end{claim}
\begin{proof}
Follows from the definitions and from Claim \ref{cla:quFreeTpEquiv}.\end{proof}
\begin{prop}
\label{prop:NIPS1} Assume $\left|S\right|=1$ and $k<\omega$. Then
there is a polynomial $p_{k}$ over $\Nn$ such that for every model
$M\models T_{S}^{\forall}$ and for every finite set $A\subseteq M$,
$\left|\left\{ M/\equiv_{A,k}\right\} \right|\leq p_{k}\left(\left|A\right|\right)$.
Moreover, we can choose $\sequence{p_{k}}{k<\omega}$ so that $p_{k}$
is linear for all $k$. \end{prop}
\begin{proof}
As $\left|S\right|=1$, we can forget the index $\eta$ and write
$<,\limb$, etc. instead of $<_{\eta},\limb_{\eta}$, etc.

Suppose $M\models T_{S}^{\forall}$. Given $a<b\in M$, the $k$-distance
between them is defined by 
\[
d_{k}\left(a,b\right)=\min\left\{ d\left(a,b\right),2k+1\right\} .
\]

Assume $a\in M$ and $A\subseteq M$ is finite.

Let $B=\cl^{\left(0\right)}\left(A\right)$ and $l=\left|B\right|$.
Recall that $l\leq f_{0}^{S}\left(\left|A\right|\right)$ where $f_{0}^{S}$
is a linear function (see Claim \ref{cla:PolBoundOnCl}). We will
divide the possible $k$-isomorphism type of $a$ over $A$ into finitely
many cases, and in each case the number of possible types will be
linear in $l$ (so linear in $\left|A\right|$). 
\begin{casenv}
\item $a\notin P$. Here there is no structure, so the number of types is
$\left|A\right|+1$. 
\item $a\in P$, and there is some $b\in B$ such that $a\leq b$. We further
divide into sub-cases:

\begin{casenv}
\item $a\in B$. In that case there are at most $l$ types. 
\item There is no $b\in B$ such that $b<a$. In that case, since $B$ is
closed under $\wedge$, $a$ is smaller than $b$ for all $b\in B$.
In this case it is enough to know the $k$-distance between $a$ and
$\lim\left(a\right)$. So there are $2k+1$ types. 
\item \label{cas:pair}There is some $b\in B$ such that $b<a$. Choose
$b_{0},b_{1}\in B$ such that $b_{1}$ is minimal with the property
that $a<b_{1}$ and $b_{0}$ is maximal such that $b_{0}<a$. Since
$B$ can also be viewed as a finite graph-theoretic tree and as such
has $l-1$ edges, we have at most $l-1$ such pairs. 

\begin{casenv}
\item $\lim\left(b_{0}\right)<\lim\left(b_{1}\right)$. Note that it follows
that $\lim\left(b_{1}\right)=b_{1}$. 

\begin{casenv}
\item $\lim\left(b_{0}\right)<\lim\left(a\right)$. Then the type is determined
by the $k$-distance between $a$ and $\lim\left(a\right)$, so there
are at most $2k+1$ types here. 
\item $\lim\left(b_{0}\right)=\lim\left(a\right)$. The type is determined
by the $k$-distance between $a$ and $b_{0}$, so again there are
at most $2k+1$ types.
\end{casenv}
\item $\lim\left(b_{0}\right)=\lim\left(b_{1}\right)$. In this case $\lim\left(b_{1}\right)=\lim\left(a\right)$.
The type is determined by the $k$-distance between $a$ and $b_{0}$
and the $k$ distance between $a$ and $b_{1}$. So totally there
are at most $4k+2$ types. 
\end{casenv}

So in this case (Case \ref{cas:pair}) there are at most $\left(l-1\right)\cdot\left(4k+2\right)$
many types. 

\end{casenv}
\item \label{cas:new branch}$a\in P$, and there is no $b\in B$ such that
$a\leq b$. Let $a'=\max\set{a\wedge b}{b\in B}$. Since there is
some $b\in B$ such that $a'\leq b$, the number of possible $k$-isomorphism
types of $a'$ over $A$ is bounded by $h\left(l\right)$ where $h$
is a linear map. Fix $\tp_{k}\left(a'/A\right)$. 

\begin{casenv}
\item $\lim\left(a\right)=\lim\left(a'\right)$. Here the type is determined
by the $k$-distance between $a$ and $a'$, so there are at most
$2k+1$ types.
\item $\lim\left(a\right)>\lim\left(a'\right)$. Here the type is determined
by the $k$-distance between $a$ and $\lim\left(a\right)$, so there
are at most $2k+1$ types. 
\end{casenv}

So in this case (Case \ref{cas:new branch}) there are at most $h\left(l\right)\cdot\left(4k+2\right)$
types. 

\end{casenv}
\end{proof}
\begin{defn}
Let $S$ be a finite standard tree, and $n<\omega$. Say that $S$
is\emph{ $n$-nice }if there is a number $N<\omega$ and a sequence
of polynomials $\sequence{p_{k}^{S}}{k<\omega}$ over $\Nn$, \uline{whose
degrees are bounded by $N$} such that for every model $M\models T_{S}^{\forall}$
and finite $A\subseteq M$, $\left|\left\{ M^{n}/\equiv_{A,k}\right\} \right|\leq p_{k}^{S}\left(\left|A\right|\right)$.
Say that $S$ is nice if it is\emph{ $n$-nice} for all $n<\omega$.
\end{defn}
From Proposition \ref{prop:NIPS1} we get:
\begin{cor}
If $\left|S\right|=1$, then $S$ is $1$-nice. \end{cor}
\begin{lem}
\label{lem:same degree continues}Suppose $S$ is a $1$-nice finite
standard tree. Then it is nice. \end{lem}
\begin{proof}
We may restrict our attention to models of $T_{S}$ (i.e., existentially
closed models of $T_{S}^{\forall}$), since every model of $T_{S}^{\forall}$
extends to a model of $T_{S}$, and the number of $k$-isomorphism
types can only increase. 

The proof is by induction on $n$. For $n=1$ this is the assumption,
so assume it holds for every $l\leq n$. Fix some polynomials $\sequence{p_{k,l}}{k<\omega,0<l\leq n}$
that witness $l$-niceness for all $l\leq n$. We will show that the
polynomials defined by $p_{k,n+1}^{S}\left(X\right)=p_{k',n}\left(X\right)\cdot p_{k,1}\left(X+1\right)$
with $k'=m_{2}\left(k,n,S\right)$ (see Lemma \ref{lem:quElLemma})
bound the number of $k$-isomorphism types. By induction, their degree
is bounded by a constant number, regardless of $k$.

We use Claim \ref{cla:TypesOver}, namely that we can identify the
number of $k$-isomorphism types and the number of $k$-types (see
Definition \ref{def:k-type}). 

Suppose $A$ is a finite subset of a model $M\models T_{S}$. For
every $k,m<\omega$ let $\Delta_{k}^{m}=\Delta_{k}^{\bar{x}\bar{y}}$
where $\lg\left(\bar{x}\right)=m$ and $\lg\left(\bar{y}\right)=\left|A\right|$.
Let $Q=S_{\Delta_{k}^{n+1}}\left(A\right)$. For each type $r\in Q$,
choose a realization $\left(\bar{a}_{r},b_{r}\right)$ where $\lg\left(\bar{a}_{r}\right)=n$.
Let $E$ be the equivalence relation on $Q$ defined by $r\mathrela Er'$
iff $b_{r}\equiv_{A,k'}b_{r'}$. Without loss of generality, for all
$r,r'\in Q$, if $r\mathrela Er'$ then $b_{r}=b_{r'}$: choose representatives
$\sequence{r_{i}}{i<l}$ for all the $E$-classes. Fix some $i<l$
and $r\mathrela Er_{i}$. Enumerate $A$ as $\bar{a}$. Since $b_{r}\bar{a}\equiv_{k'}b_{r_{i}}\bar{a}$,
by Lemma \ref{lem:quElLemma} there is some $\bar{a}_{r}'\in M^{n}$
such that $\bar{a}_{r}b_{r}\bar{a}\equiv_{k}\bar{a}_{r}'b_{r_{i}}\bar{a}$,
i.e., $\bar{a}_{r}b_{r}\equiv_{A,k}\bar{a}_{r}'b_{r_{i}}$, so we
can replace $\left(\bar{a}_{r},b_{r}\right)$ by $\left(\bar{a}_{r}',b_{r_{i}}\right)$.
Now for each $E$-equivalence class $C\subseteq Q$, the map $r\mapsto\tp_{k}^{S}\left(\bar{a}_{r}/A\cup\left\{ b_{r}\right\} \right)$
from $C$ to $S_{\Delta_{k}^{n}}\left(A\cup\left\{ b_{r}\right\} \right)$
is injective, so $\left|C\right|\leq p_{k,n}^{S}\left(\left|A\right|+1\right)$.
The number of $E$-classes is bounded by $p_{k',1}^{S}\left(\left|A\right|\right)$,
so we are done. \end{proof}
\begin{thm}
\label{thm:T is nice} Suppose $S$ is a finite standard tree. Then
it is nice. \end{thm}
\begin{proof}
The proof is by induction on $\left|S\right|$. For $\left|S\right|=1$
it follows from Proposition  \ref{prop:NIPS1} and Lemma \ref{lem:same degree continues}
(and for $\left|S\right|=0$ it is obvious).

Assume $1<\left|S\right|$. By Lemma \ref{lem:same degree continues},
it is enough to show that $S$ is $1$-nice. 

Let $\eta_{0}$ be the root of $S$, $S_{0}=\left\{ \eta_{0}\right\} $
and let $S=\bigcup\left\{ S_{i}\left|\, i<m\right.\right\} $ where
for $1\leq i<m$ the $S_{i}$'s are the connected components of $S$
above $\eta_{0}$. For $i\leq m$, let $P_{S_{i}}=\bigvee\left\{ P_{\eta}\left|\,\eta\in S_{i}\right.\right\} $.
For $i<m$, let $\eta_{i}=\min\left(S_{i}\right)$. Suppose $\sequence{p_{k,n}^{i}}{k,n<\omega,i<m}$
witness that $S_{i}$ are nice. Suppose the degree of $p_{k,n}^{i}$
is bounded by $N_{n}$ for all $k,n<\omega$ and $i<m$. We may assume
that $p_{k,n}^{i}\leq p_{k,n+1}^{i}$ and $N_{n}\leq N_{n+1}$ for
all $k,n<\omega$ and $i<m$. 

Assume $A\subseteq M\models T_{S}^{\forall}$ is finite and $a\in M$.
We will divide the possible $k$-isomorphism types of $a$ over $A$
into finitely many cases. In each case we will have a polynomial bound
(in terms of $\left|A\right|$) on the number of types. This polynomial
will have degree at most $m\cdot N_{K}$ where $K=3$. Since $M,A$
and $a$ were arbitrary this will show that $S$ is $1$-nice. 

Let $A_{i}=\cl^{\left(k\right)}\left(A\right)\cap P_{S_{i}}^{M}$. 
\begin{casenv}
\item $a\notin P_{\eta}^{M}$ for all $\eta\in S$. In that case there are
at most $\left|A\right|+1$ types. 
\item $a\in P_{\eta_{i}}^{M}$ for some $1\leq i<m$. It is enough to determine
$\tp_{k}^{S_{i}}\left(a/A_{i}\right)$. If $\tp_{k}^{S_{i}}\left(a/A_{i}\right)=\tp_{k}^{S_{i}}\left(b/A_{i}\right)$,
then $a\equiv_{A_{i},k}b$ (by Claim \ref{cla:TypesOver}), so there
is an $L_{S_{i}}'$ isomorphism $f':\cl^{\left(k\right)}\left(A_{i}a\right)\to\cl^{\left(k\right)}\left(A_{i}b\right)$
taking $a$ to $b$ and fixing $A_{i}$. Define $f:\cl^{\left(k\right)}\left(Aa\right)\to\cl^{\left(k\right)}\left(Ab\right)$
by 
\[
\left(f'\upharpoonright\cl^{\left(k\right)}\left(A\cup\left\{ a\right\} \right)\cap P_{S_{i}}^{M}\right)\cup\left(\id\upharpoonright\cl^{\left(k\right)}\left(A\right)\right).
\]
This is an isomorphism. Now, note that $\left|A_{i}\right|\leq f_{k}^{S_{i}}\left(\left|A\right|\right)$
which is linear in $\left|A\right|$ (see Claim \ref{cla:PolBoundOnCl}),
and the number of types over $A_{i}$ is bounded by $p_{k,1}^{i}\left(\left|A_{i}\right|\right)\leq p_{k,1}^{i}\left(f_{k}^{S_{i}}\left(\left|A\right|\right)\right)$. 
\item $a\in P_{\eta_{0}}$. Let $B=A\cap P_{\eta_{0}}^{M}$. First we determine
$\tp_{k}^{S_{0}}\left(a/B\right)$, for this we have at most $p_{k,1}^{0}\left(\left|A\right|\right)$
many possibilities. Fix one such type.

Suppose $a\equiv_{B,k}^{S_{0}}b$. Let $f'$ be an $L_{S_{0}}'$-isomorphism
such that $B\cup\left\{ a\right\} \xrightarrow[k]{S_{0},f'}B\cup\left\{ b\right\} $,
$f'$ fixes $B$ and takes $a$ to $b$. Let $F=\cl^{\left(k\right)}\left(A\cup\left\{ a\right\} \right)\cap P_{\eta_{0}}^{M}$
and $F'=f'\left(F\right)$, so that $f$ is an $L_{S_{0}}'$ isomorphism
between $F$ and $F'$. By Claim I (1) in the proof of Lemma \ref{lem:quElLemma},
there are at most $K$ (i.e., $3$)  $\sim_{\eta_{0}}^{F}$-classes
in $F$ that are not already in $\cl^{\left(k\right)}\left(A\right)$,
suppose there are $K'\leq K$ such classes. Let $\bar{b}$ be an enumeration
of $B$, and $\bar{y}$ a tuple of variables of the same length. If
$\sequence{t_{i}\left(x,\bar{y}\right)}{i<K'}$ are terms from $\cl^{\left(k\right),S_{0}}\left(x\bar{y}\right)$
such that the new classes are exactly $\set{\left[t_{i}\left(a,\bar{b}\right)\right]_{\sim_{\eta_{0}}^{F}}}{i<K'}$,
then the new classes in $F'$ are $\set{\left[t_{i}\left(b,\bar{b}\right)\right]_{\sim_{\eta_{0}}^{F'}}}{i<K'}$.
This means that we can fix such terms depending only on $\tp_{k}^{S_{0}}\left(a/B\right)$.
Now it is enough to determine $\tp_{k}^{S_{i}}\left(\sequence{G_{\eta_{0},\eta_{i}}\left(t_{l}\left(a,\bar{b}\right)\right)}{l<K'}/A_{i}\right)$
for each $1\leq i<m$.

Indeed, suppose that $a,b$ and $f'$ are as above and moreover for
each $1\leq i<m$, $\sequence{G_{\eta_{0},\eta_{i}}\left(t_{l}\left(a,\bar{b}\right)\right)}{l<K'}\equiv_{k,A_{i}}\sequence{G_{\eta_{0},\eta_{i}}\left(t_{l}\left(b,\bar{b}\right)\right)}{l<K'}$.
Let $g_{i}$ be an $L_{S_{i}}'$-isomorphism fixing $A_{i}$ witnessing
this. Then 
\[
\id\upharpoonright\cl^{\left(k\right)}\left(A\right)\cup f'\cup\bigcup_{1\leq i<m}\left(g_{i}\upharpoonright\cl^{\left(k\right)}\left(A\cup\left\{ a\right\} \right)\cap P_{S_{i}}^{M}\right)
\]
 is an $L_{S}'$-isomorphism showing that $a\equiv_{A}^{k}b$. This
follows from the fact that if $e\sim_{\eta_{0}}^{F}e'$ then $G_{\eta_{0},\eta_{i}}\left(e\right)=G_{\eta_{0},\eta_{i}}\left(e'\right)$.

In this case there are at most $p_{k,1}^{0}\left(\left|A\right|\right)\cdot\prod_{1\leq i<m}p_{k,K}^{i}\left(f_{k}^{S_{i}}\left(\left|A\right|\right)\right)$
types (here we used the assumption that $p_{k,K'}^{i}\leq p_{k,K}^{i}$). 

\end{casenv}
\end{proof}
\begin{cor}
\label{cor:Strongly dep T_S}Suppose $S$ is a finite standard tree.
Then $T_{S}$ is strongly$^{2}$-dependent.\end{cor}
\begin{proof}
We will apply Lemma \ref{lem:strong dep. pol bound}.

Let $\Delta\left(\bar{x};\bar{y}\right)$ be a finite set of formulas.
By quantifier elimination, we may assume that $\Delta$ is quantifier
free. Let $k=\max\left\{ r_{\suc}\left(\varphi\right)\left|\,\varphi\in\Delta\right.\right\} $
and $m=\left|S_{\Delta\left(\bar{x};\bar{y}\right)}\left(A\right)\right|$.
Let $\left\{ \bar{c}_{i}\left|\, i<m\right.\right\} $ be a set of
tuples satisfying all the different types in $S_{\Delta\left(\bar{x};\bar{y}\right)}\left(A\right)$
in some model $M$ of $T_{S}$. If $i\neq j$ then $\tp_{k}\left(\bar{c}_{i}/A\right)\neq\tp_{k}\left(\bar{c}_{j}/A\right)$
(by Claim \ref{cla:TypesOver}), so $m\leq\left|\left\{ M^{\lg\left(\bar{x}\right)}/\equiv_{A,k}\right\} \right|$,
and hence we are done by Theorem \ref{thm:T is nice}. 
\end{proof}
So far we mostly assumed that $S$ is finite. Now we will let $S$
be any standard tree. 
\begin{cor}
\label{cor:Sinfinite} Suppose $S$ is a standard tree. If $M\models T_{S}^{\forall}$
then since 
\[
Th\left(M\right)=\bigcup\left\{ Th\left(M\upharpoonright L_{S_{0}}\right)\left|\, S_{0}\subseteq S\,\&\,\left|S_{0}\right|<\aleph_{0}\right.\right\} ,
\]
by Remark \ref{rem:TUnivJEP}, Corollary \ref{cor:QE} is true in
the case where $S$ is infinite. So $T_{S}$ is well defined in this
case as well and it is in fact $\bigcup\left\{ T_{S_{0}}\left|\, S_{0}\subseteq S\,\&\,\left|S_{0}\right|<\aleph_{0}\right.\right\} $.
It eliminates quantifiers and is dependent.
\end{cor}

\subsection*{Adding Constants}

We want to find an example of every cardinality, and so we add constants
to the language. For a cardinal $\theta$, the theory $T_{S}^{\theta}$
will be $T_{S}$ augmented with the quantifier free diagram of a model
of $T_{S}^{\forall}$ of cardinality $\theta$. The simplest thing
to do is to add $\theta$-many constants that do not belong to any
$P_{\eta}$. The problem with this approach is that the induction
would not work in the proof of the main theorem. So instead we put
a tree of constants in every $P_{\eta}$. Formally:
\begin{defn}
\label{def:T_S^Theta}Let $S$ be a standard tree. For a cardinal
$\theta$, let $L_{S}^{\theta}=L_{S}\cup\left\{ e_{\eta,i}\left|\, i<\theta,\eta\in S\right.\right\} $
where $\left\{ e_{\eta,i}\left|\, i<\theta,\eta\in S\right.\right\} $
are new constants. Let $T_{S}^{\forall,\theta}$ be the theory $T_{S}^{\forall}$
with the axioms stating that for all $\eta,\eta_{1},\eta_{2}\in S$
and $i,j,i',j'<\theta$ such that $\eta_{1}<_{\suc}\eta_{2}$,\end{defn}
\begin{itemize}
\item $e_{\eta,i}\in P_{\eta}$,
\item $i\neq j\Rightarrow e_{\eta,i}\neq e_{\eta,j}$,
\item $i\neq j,i'\neq j'\Rightarrow e_{\eta,i}\wedge_{\eta}e_{\eta,j}=e_{\eta,i'}\wedge_{\eta}e_{\eta,j'}$,
\item $\eta_{1}<_{\suc}\eta_{2}\Rightarrow G_{\eta_{1},\eta_{2}}\left(e_{\eta_{1},i}\right)=e_{\eta_{2},i}$,
\item $\minb_{\eta}\left(e_{\eta,i}\wedge e_{\eta,j}\right)=e_{\eta,i}\wedge e_{\eta,j}$
and
\item $\suc_{\eta}\left(e_{\eta,i}\wedge e_{\eta,j},e_{\eta,i}\right)=e_{\eta,i}$.\end{itemize}
\begin{cor}
\label{cor:TSThetaModelCompletion} Suppose $S$ is a standard tree.
\begin{enumerate}
\item $T_{S}^{\forall,\theta}$ has JEP and AP.
\item $T_{S}^{\forall,\theta}$ has a model completion --- $T_{S}^{\theta}$
--- that is complete, dependent and has quantifier elimination.
\item Given any model $M\models T_{S}^{\forall}$, there is a model $M'\models T_{S}^{\forall,\theta}$
satisfying $M'\upharpoonright L_{S}\supseteq M$.
\item If $S$ is finite then $T_{S}^{\theta}$ is strongly$^{2}$ dependent. 
\end{enumerate}
\end{cor}
\begin{proof}
(1) This follows from Corollary \ref{cor:Model completion} (noting
that JEP for $T_{S}^{\forall,\theta}$ follows from AP for $T_{S}^{\forall}$). 

(2) Since $T_{S}$ is the model completion of $T_{S}^{\forall}$ and
$T_{S}^{\forall,\theta}$ is the quantifier free diagram of a model
of $T_{S}^{\forall}$, $T_{S}^{\theta}=T_{S}\cup T_{S}^{\forall,\theta}$
is a complete theory. Since we only added constants, $T_{S}^{\theta}$
is dependent and has quantifier elimination. 

(3) This follows from JEP for $T_{S}^{\forall}$. 

(4) This follows from Corollary \ref{cor:Strongly dep T_S}. 
\end{proof}

\section{\label{sec:The-inaccessible-case}The inaccessible case}

In this section we will deal with the main technical obstacle in proving
Main Theorem \ref{mainthmA}. The proof, which will be described in
Section \ref{sec:The-main-theorem}, is by induction in the following
sense: for $\Ss=2^{<\omega}$, cardinals $\kappa,\theta$ and a limit
ordinal $\delta\geq\omega$ such that $\kappa\not\to\left(\delta\right)_{\theta}^{<\omega}$,
we will find a model $M\models T_{\Ss}^{\forall,\theta}$ and a set
$A\subseteq P_{\left\langle \right\rangle }^{M}$ of size $\left|A\right|\geq\kappa$
with no non-constant indiscernible sequence in $A^{\delta}$. We are
allowed to use induction since $\lambda\not\to\left(\delta\right)_{\theta}^{<\omega}$
for all $\lambda<\kappa$. We divide into cases, namely $\kappa\leq\theta$
, $\kappa$ singular and $\kappa$ regular but not strongly inaccessible.
The main problem is in the remaining case, i.e., when $\kappa$ is
strongly inaccessible. In all other cases, the proof will follow by
induction without using explicitly the fact that $\kappa\not\to\left(\delta\right)_{\theta}^{<\omega}$.
\begin{assumption}
\label{ass:NoColoring} Assume for this section that $\theta<\kappa$
are cardinals, $\delta\geq\omega$ is a limit ordinal and that $\kappa$
is strongly inaccessible such that $\kappa\not\to\left(\delta\right)_{\theta}^{<\omega}$.
\end{assumption}
This section is divided into two subsections. 

In the first subsection we define a class $\mathcal{T}$ of models
of $T_{\omega}^{\forall,\theta}$ (here $S=\omega$, with the tree
structure being the usual order on $\omega$). We will analyze sequences
of elements in models in $\mathcal{T}$ that are close to being indiscernible.
There are two main results here, the first (Proposition \ref{prop:Dichotomy-1-1})
says that sequences (of singletons) that are closed to being indiscernible
can have two forms: ``almost increasing'' and ``fan''. ``Almost
increasing'' means that $s_{i}\wedge s_{i+1}<s_{i+1}\wedge s_{i+2}$,
and ``fan'' means that $s_{i}\wedge s_{j}$ is constant. The second
result (Corollary \ref{cor:TheHFunction-1-1}) deals with applying
a specific definable map on sequences. Given an almost increasing
sequence $\bar{s}$, let $H\left(\bar{s}\right)=\bar{t}$ where $t_{i}=G\left(\suc\left(\minb\left(s_{i}\wedge s_{i+1}\right),s_{i+1}\right)\right)$
(where $G$ is some $G_{n,n+1}$, recall that here $S=\omega$). We
will show that if applying $H$ again and again we always get an almost
increasing sequence, then this almost increasing sequence will satisfy
$\suc\left(\minb\left(t_{i}\wedge t_{i+1}\right),t_{i}\right)=t_{i}$.

In the second subsection we will construct a model in $\mathcal{T}$
that uses explicitly a witness of $\kappa\not\to\left(\delta\right)_{\theta}^{<\omega}$.
For this model, $P_{0}=\kappa$. We will show, applying the analysis,
that if we have an indiscernible sequence in $P_{0}$ such that applying
$H$ to it again and again results in almost increasing sequences,
then there is a homogeneous sub-sequence of $\kappa$ of length $\delta$,
contradicting the assumption. So after applying $H$ finitely many
times we must get a fan. This model will come equipped with equivalence
relations on the trees $P_{n}$, which refines the neighboring relation
($x,y$ are neighbors if they succeed the same element). The point
is that the number of classes inside a given neighborhood will be
less than $\kappa$. This will enable us to use the induction hypothesis
in the proof of Main Theorem \ref{mainthmA}. 

The models in $\mathcal{T}$ will be standard in the following sense:
\begin{defn}
\label{def:standard model} Suppose $S$ is a standard tree. Call
a model of $T_{S}^{\forall}$ \emph{standard} if for every $\eta\in S$,
$\left(P_{\eta},<_{\eta}\right)$ is a standard tree, and $\wedge_{\eta},\lim_{\eta},\suc_{\eta}$
are all interpreted in the natural way (so $\lim_{\eta}\left(a\right)$
is the greatest element $\leq a$ of a limit level). 
\end{defn}
Let us fix some notation:
\begin{notation}
\label{not: what is indiscernible, modulo}Suppose $S$ is the standard
tree $\omega$ with the usual ordering. Assume $M\models T_{S}^{\forall}$
and $x,y\in P_{\eta}^{M}$.
\begin{enumerate}
\item When we say indiscernible, we shall always mean indiscernible for
quantifier free formulas.
\item We say that $x\equiv0\modp{\omega}$ when $x=\limb\left(x\right)$.
For $n<\omega$, we say that $x\equiv n+1\modp{\omega}$ where $x\neq\limb_{\eta}\left(x\right)$
and $\pre_{\eta}\left(x\right)\equiv n\modp{\omega}$. Note that for
a fixed $n$, the set $\left\{ x\left|\, x\equiv n\modp{\omega}\right.\right\} $
is quantifier free definable. In addition, if $M$ is standard, then
for every $x$ there is some $n<\omega$ such that $x\equiv n\modp{\omega}$
(where $n$ is the unique number satisfying $\lev\left(x\right)=\alpha+n$
for a limit ordinal $\alpha$).
\item Say that $x\equiv y\modp{\omega}$ if there is $n<\omega$ such that
$x\equiv n\modp{\omega}$ and $y\equiv n\modp{\omega}$.
\item Instead of $G_{n,n+1}$ we write $G_{n}$.
\end{enumerate}
\end{notation}

\subsection*{Analysis of indiscernibles in $\Tt$}
\begin{defn}
\label{def:ClassT}Let $\mathcal{T}$ be the class of models $M\models T_{\omega}^{\forall}$
that satisfy:
\begin{enumerate}
\item $M$ is standard (see Definition \ref{def:standard model}). 
\item \label{enu:GDecreasing} For $t\in P_{n}$, $\lev\left(G_{n}\left(t\right)\right)\leq\lev\left(t\right)$.
\item \label{enu:imageSuc}$G_{n}:\Suc\left(P_{n}\right)\to\Suc\left(P_{n}\right)$
(i.e., we demand that the image is also a successor).
\item \label{enu:GInjective} If $\left\langle s_{i}\left|\, i<\delta\right.\right\rangle $
is an increasing sequence in $\Suc\left(P_{n}\right)$ such that $s_{i}\equiv s_{j}\modp{\omega}$
for all $i<j<\delta$ then $i<j\Rightarrow G_{n}\left(s_{i}\right)\neq G_{n}\left(s_{j}\right)$.
\end{enumerate}
\end{defn}
\begin{notation}
\label{not:neighbors}For $M\in\mathcal{T}$ and $n<\omega$,
\begin{enumerate}
\item \label{enu:Normal} We say that $s,t\in P_{n}^{M}$ are neighbors,
denoted by $t\mathrela{E^{\nb}}s$ when $\left\{ x\left|\, x<t\right.\right\} =\left\{ x\left|\, x<s\right.\right\} $.
This is an equivalence relation. As $P_{n}$ is a normal tree, for
$t$ of a limit level its $E^{\nb}$-class is $\left\{ t\right\} $.
\item Let $\Suc\left(M\right)=\bigcup\left\{ \Suc\left(P_{n}^{M}\right)\left|\, n<\omega\right.\right\} $.
\item $\bar{s}$, $\bar{t}$ and $\bar{r}$ will denote $\delta$-sequences,
e.g., $\bar{s}=\left\langle s_{i}\left|\, i<\delta\right.\right\rangle $.
\item If $\bar{s}$ is contained in some $P_{n}^{M}$ and $n$ is clear
from the context or insignificant, then we write $<$ instead of $<_{n}$
etc.
\end{enumerate}
\end{notation}
\begin{defn}
Recall that given $\delta^{*}\geq\omega$ and an indiscernible sequence
$\bar{s}=\sequence{s_{i}}{i<\delta^{*}}$, its quantifier free Ehrenfeucht-Mostowski
type (or in short quantifier free EM-type) is defined as $\sequence{\tp_{\qf}\left(s_{0},\ldots,s_{n-1}\right)}{n<\omega}$.
In general, a quantifier free EM-type is a sequence $\bar{p}=\sequence{p_{n}}{n<\omega}$
such that $p_{n}\in S_{n}^{\qf}\left(\emptyset\right)$. 
\end{defn}
We need the following generalization of indiscernible sequences for
$\mathcal{T}$:
\begin{defn}
\label{def:NI-1-1} A sequence $\bar{s}=\left\langle s_{i}\left|\, i<\delta\right.\right\rangle $
is called \emph{nearly indiscernible} (in short \emph{NI}) if:
\begin{enumerate}
\item There is $n<\omega$ and an EM-type $\bar{p}=\sequence{p_{k}\in S_{k}^{\qf}\left(\emptyset\right)}{k<\omega}$
such that if $i_{0}<\cdots<i_{k-1}<\delta$ and $i_{j}+n\leq i_{j+1}$
for all $j<k$, then $\left(s_{i_{0}},\ldots,s_{i_{k-1}}\right)\models p_{k}$.
(So for $\delta^{*}\leq\delta$ every sub-sequence $\left\langle s_{i_{j}}\left|\, j<\delta^{*}\right.\right\rangle $
with $i_{j}+n\leq i_{j+1}<\delta$ is indiscernible and its quantifier
free EM-type is $\bar{p}$.) We call this property \emph{sparseness}.
\item For $i,j<\delta$ and $k<\omega$, $\tp_{\qf}\left(s_{i},\ldots,s_{i+k}\right)=\tp_{\qf}\left(s_{j},\ldots,s_{j+k}\right)$.
We call this property \emph{sequential homogeneity}.
\end{enumerate}
\end{defn}

\begin{defn}
\label{def:HNI-1-1} A sequence $\bar{s}=\left\langle s_{i}\left|\, i<\delta\right.\right\rangle $
is called \emph{hereditarily nearly Indiscernible} (in short \emph{HNI})
if: 

For every term $\sigma\left(x_{0},\ldots,x_{n-1}\right)$, the sequence
$\bar{t}=\left\langle t_{i}\left|\, i<\delta\right.\right\rangle $
defined by $t_{i}=\sigma\left(s_{i},\ldots,s_{i+n-1}\right)$ is NI.\end{defn}
\begin{rem}
If $\bar{s}$ is HNI then it is NI, and for every term $\sigma\left(x_{0},\ldots,x_{n-1}\right)$,
the sequence $\bar{t}=\left\langle t_{i}\left|\, i<\delta\right.\right\rangle $
defined by $t_{i}=\sigma\left(s_{i},\ldots,s_{i+n-1}\right)$ is HNI.
Indeed, for any term $\tau\left(x_{0},\ldots,x_{k-1}\right)$, let
\[
\tau'\left(x_{0},\ldots,x_{n+k-2}\right)=\tau\left(\sigma\left(x_{0},\ldots,x_{n-1}\right),\ldots,\sigma\left(x_{k-1},\ldots,x_{n+k-2}\right)\right),
\]
then the sequence $\bar{r}=\sequence{r_{i}}{i<\delta}$ defined by
$r_{i}=\tau\left(t_{i},\ldots,t_{i+k-1}\right)$ is equal to $\tau'\left(s_{i},\ldots,s_{i+n+k-2}\right)$
thus it is NI. \end{rem}
\begin{example}
\label{exa:Indiscernible are HNI}If $\bar{s}=\left\langle s_{i}\left|\, i<\delta\right.\right\rangle $
is indiscernible, then it is HNI.\end{example}
\begin{proof}
Suppose $\sigma\left(x_{0},\ldots,x_{n-1}\right)$ is a term. If $t_{i}=\sigma\left(s_{i},\ldots,s_{i+n-1}\right)$,
then any sub-sequence of $\bar{t}=\left\langle t_{i}\left|\, i<\delta\right.\right\rangle $
where the distance between two consecutive elements is at least $n$
is an indiscernible sequence with a constant quantifier free EM-type.
This shows sparseness. 

For sequential homogeneity, note that for a quantifier free formula
$\varphi$, 
\[
\varphi\left(t_{i},\ldots,t_{i+k}\right)=\varphi\left(\sigma\left(s_{i},\ldots,s_{i+n-1}\right),\ldots,\sigma\left(s_{i+k},\ldots,s_{i+n+k-1}\right)\right).
\]

Let $i,j<\delta$. As $\tp_{\qf}\left(s_{i},\ldots,s_{i+n+k-1}\right)=\tp_{\qf}\left(s_{j},\ldots,s_{j+n+k-1}\right)$,
it follows that 
\[
\tp_{\qf}\left(t_{i},\ldots,t_{i+k}\right)=\tp_{\qf}\left(t_{j},\ldots,t_{j+k}\right).
\]
\end{proof}
\begin{defn}
\label{def:ind-1} Assume $M\in\mathcal{T}$.
\begin{enumerate}
\item $\ind\left(M\right)$ is the set of all non-constant indiscernible
sequences $\bar{s}\in\Suc\left(M\right)^{\delta}$.
\item $\HNind\left(M\right)$ is the set of all non-constant HNI sequences
$\bar{s}\in\Suc\left(M\right)^{\delta}$.
\item $\ai\left(M\right)$ is the set of sequences $\bar{s}$ such that
for some $n<\omega$, $\bar{s}\in\left(P_{n}^{M}\right)^{\delta}$
and $s_{i}\wedge s_{i+1}<s_{i+1}\wedge s_{i+2}$ (ai stands for ``almost
increasing'', note that if $\bar{s}$ is increasing then it is here).
\item $\ind_{f}\left(M\right)$ is the set of all sequences $\bar{s}\in\ind\left(M\right)$
such that $s_{i}\wedge s_{j}$ is constant for all $i<j<\delta$ (f
stands for ``fan'').
\item $\ind_{i}\left(M\right)$ is the set of all increasing sequences $\bar{s}\in\ind\left(M\right)$.
\item $\ind_{\ai}\left(M\right)=\ind\left(M\right)\cap\ai\left(M\right)$.
\item Define $\HNind_{f}\left(M\right)$, $\HNind_{i}\left(M\right)$ and
$\HNind_{\ai}\left(M\right)$ similarly, but we demand that the sequences
are HNI.
\end{enumerate}
\end{defn}
From now on, assume $M\in\mathcal{T}$.
\begin{rem}
\label{rem:AIncreasingWedge} If $\bar{s}\in\ai\left(M\right)$, then
$s_{i}\wedge s_{i+n}=s_{i}\wedge s_{i+1}$ for all $2\leq n<\omega$
and $i<\delta$ (prove by induction on $n$, using the fact that if
$a\wedge b<b\wedge b'$ then $a\wedge b'=a\wedge b$). \end{rem}
\begin{prop}
\label{prop:Dichotomy-1-1} $\HNind\left(M\right)=\HNind_{\ai}\left(M\right)\cup\HNind_{f}\left(M\right)$.\end{prop}
\begin{proof}
Assume that $\bar{s}\in\HNind\left(M\right)$. Since $\bar{s}$ is
NI, there is some $n<\omega$ that witnesses sparseness. As for $i<j<k$,
$s_{i}\wedge s_{j}$ is comparable with $s_{j}\wedge s_{k}$, by Ramsey
there is an infinite subset $A\subseteq\omega$ that satisfies one
of the following possibilities:
\begin{enumerate}
\item For all $i<j<k\in A$, $s_{i}\wedge s_{j}=s_{j}\wedge s_{k}$, or
\item For all $i<j<k\in A$, $s_{i}\wedge s_{j}<s_{j}\wedge s_{k}$.
\end{enumerate}
(note that it cannot be that $s_{j}\wedge s_{k}<s_{i}\wedge s_{j}$
because the trees are well ordered).

Assume (1) is true.

It follows that if $i<j<k<l\in A$ then $s_{i}\wedge s_{j}=s_{j}\wedge s_{k}=s_{k}\wedge s_{l}$.
If $n\leq j-i,k-j,l-k$, then by the choice of $n$, the same is true
for all $i<j<k<l<\delta$ where the distances are at least $n$. Moreover,
given $i<j,k<l$ such that $n\leq j-i$ and $n\leq l-k$, then $s_{i}\wedge s_{j}=s_{\max\left\{ j,l\right\} +n}\wedge s_{\max\left\{ j,l\right\} +2n}$,
and the same is true for $s_{k}\wedge s_{l}$. It follows that $s_{i}\wedge s_{j}=s_{k}\wedge s_{l}$.

Choose some $0<i<n$. 

Assume for contradiction that $s_{0}\wedge s_{i}<s_{i}\wedge s_{2i}$,
then by sequential homogeneity $\left\langle s_{i\alpha}\left|\,\alpha<\delta\right.\right\rangle \in\ai\left(M\right)$.
In this case, by Remark \ref{rem:AIncreasingWedge}, $s_{0}\wedge s_{i}<s_{i}\wedge s_{2i}=s_{i}\wedge s_{ni+i}$.
But $s_{0}\wedge s_{ni+i}=s_{i}\wedge s_{ni+i}$, and so on the one
hand $s_{0}\wedge s_{i}<s_{0}\wedge s_{ni+i}$, and on the other hand
$s_{0}\wedge s_{ni+i}\leq s_{i}$ --- together it's a contradiction.

It cannot be that $s_{0}\wedge s_{i}>s_{i}\wedge s_{2i}$ since the
trees are well ordered.

So (again by the sequential homogeneity) it must be that $s_{0}\wedge s_{i}=s_{i}\wedge s_{2i}=\cdots=s_{ni}\wedge s_{ni+i}$.
So necessarily $s_{0}\wedge s_{i}\leq s_{0}\wedge s_{ni}$, but in
addition $s_{0}\wedge s_{ni}=s_{0}\wedge s_{ni+i}$ (since the distance
is at least $n$) and so ${s_{0}\wedge s_{i}=s_{ni}\wedge s_{ni+i}\geq s_{0}\wedge s_{ni}}$,
and hence $s_{0}\wedge s_{i}=s_{0}\wedge s_{ni}=s_{0}\wedge s_{n}$.

It follows that $s_{i_{0}}\wedge s_{i_{0}+i}=s_{i_{0}}\wedge s_{i_{0}+n}=s_{0}\wedge s_{n}$
for every $i_{0}<\delta$. This is true for all $i$ such that $i_{0}+i<\delta$
and so $s_{i}\wedge s_{j}=s_{0}\wedge s_{n}$ for all $i<j<\delta$.
So in this case $\bar{s}\in\HNind_{f}\left(M\right)$.

Assume (2) is true. Assume that $i<j<k\in A$ and the distances are
at least $n$. Then, as $s_{i}\wedge s_{j}<s_{j}\wedge s_{k}$, it
follows from sparseness that $\left\langle s_{n\alpha}\left|\,\alpha<\delta\right.\right\rangle \in\ai\left(M\right)$
and that $\left\langle s_{0},s_{n+1},s_{3n},s_{4n},\ldots\right\rangle \in\ai\left(M\right)$.
In particular, by Remark \ref{rem:AIncreasingWedge}, $s_{0}\wedge s_{n}=s_{0}\wedge s_{3n}=s_{0}\wedge s_{n+1}$.

If $s_{0}\wedge s_{1}<s_{1}\wedge s_{2}$, then $\bar{s}\in\HNind_{\ai}\left(M\right)$
by sequential homogeneity and we are done, so assume this is not the
case.

It cannot be that $s_{0}\wedge s_{1}>s_{1}\wedge s_{2}$ (because
the trees are well ordered).

Assume for contradiction that $s_{0}\wedge s_{1}=s_{1}\wedge s_{2}$.
By sequential homogeneity it follows that $s_{0}\wedge s_{1}=s_{n}\wedge s_{n+1}$.
We also know that $s_{0}\wedge s_{n}=s_{0}\wedge s_{n+1}$, and together
we have $s_{0}\wedge s_{1}=s_{0}\wedge s_{n+1}$, and again by sequential
homogeneity, $s_{n}\wedge s_{2n+1}=s_{n}\wedge s_{n+1}$, and so $s_{n}\wedge s_{2n+1}=s_{0}\wedge s_{n}$
--- a contradiction (because the distances are at least $n$).\end{proof}
\begin{defn}
\label{def:TheHFunction-1-1} Define the function $H:\HNind_{\ai}\left(M\right)\to\HNind\left(M\right)$
as follows: given $\bar{s}\in\HNind_{\ai}\left(M\right)$, let $H\left(\bar{s}\right)=\bar{t}$
where $t_{i}=G\left(\suc\left(\minb\left(s_{i}\wedge s_{i+1}\right),s_{i+1}\right)\right)$.
(Recall that $G=G_{n}$ where the sequence $\bar{s}$ is contained
in $P_{n}^{M}$.)\end{defn}
\begin{rem}
$H$ is well defined: if $\bar{s}\in\HNind_{\ai}\left(M\right)$ then
$H\left(\bar{s}\right)$ is in $\HNind\left(M\right)$. This is because
$\bar{t}=H\left(\bar{s}\right)$ is not constant --- by Clause (\ref{enu:GInjective})
of Definition \ref{def:ClassT} (it is applicable: the sequence $\left\langle s_{i}\wedge s_{i+1}\left|\, i<\delta\right.\right\rangle $
is NI and increasing, so there is some $n<\omega$ such that $s_{i}\wedge s_{i+1}\equiv n\modp{\omega}$
for all $i<\delta$, and hence $\left\langle \limb\left(s_{i}\wedge s_{i+1}\right)\left|\, i<\delta\right.\right\rangle $
is increasing).
\end{rem}
As usual, we denote $H^{\left(0\right)}\left(\bar{s}\right)=\bar{s}$
and $H^{\left(n\right)}\left(\bar{s}\right)=H\left(H^{\left(n-1\right)}\left(\bar{s}\right)\right)$
for $n>0$. 
\begin{cor}
\label{cor:TheHFunction-1-1} Let $\bar{s}\in\HNind_{\ai}\left(M\right)$.
If for no $n<\omega$, $H^{\left(n\right)}\left(\bar{s}\right)\in\HNind_{f}\left(M\right)$,
then for all $n<\omega$, $H^{\left(n\right)}\left(\bar{s}\right)\in\HNind_{\ai}\left(M\right)$.
Moreover, in this case there exists some $K<\omega$ such that for
all $n\geq K$, if $\bar{t}=H^{\left(n\right)}\left(\bar{s}\right)$
then $\suc\left(\minb\left(t_{i}\wedge t_{i+1}\right),t_{i}\right)=t_{i}$.\end{cor}
\begin{proof}
By Proposition \ref{prop:Dichotomy-1-1}, it follows by induction
on $n<\omega$ that $H^{\left(n\right)}\left(\bar{s}\right)\in\HNind_{\ai}\left(M\right)$
and so $H^{\left(n+1\right)}\left(\bar{s}\right)$ is well defined.

For $n<\omega$, let $\bar{s}_{n}=H^{\left(n\right)}\left(\bar{s}\right)$,
and let us enumerate this sequence as $\bar{s}_{n}=\left\langle s_{n,i}\left|\, i<\delta\right.\right\rangle $.

$\lev\left(\minb\left(s_{n,0}\wedge s_{n,1}\right)\right)<\lev\left(s_{n,0}\right)$
because $\lev\left(s_{n,0}\right)$ is a successor ordinal (by Clause
(\ref{enu:imageSuc}) of Definition \ref{def:ClassT}) while $\lev\left(\minb\left(x\right)\right)$
is a limit ordinal for all $x\in M$. 

So $\lev\left(\suc\left(\minb\left(s_{n,0}\wedge s_{n,1}\right)\right),s_{n,1}\right)\leq\lev\left(s_{n,0}\right)$,
and so by Clause (\ref{enu:GDecreasing}) of Definition \ref{def:ClassT},
\[
\left\langle \lev\left(s_{n,0}\right)\left|\, n<\omega\right.\right\rangle 
\]
 is a $\leq$-decreasing sequence.

Hence there is some $K<\omega$ and some $\alpha$ such that $\lev\left(s_{n,0}\right)=\alpha$
for all $K\leq n$. Assume without loss of generality that $K=0$.

Let $n<\omega$. We know that
\begin{eqnarray*}
\lev\left(s_{n+1,0}\right)\leq\lev\left(\suc\left(\minb\left(s_{n,0}\wedge s_{n,1}\right),s_{n,1}\right)\right) & =\\
 & = & \lev\left(\suc\left(\minb\left(s_{n,0}\wedge s_{n,1}\right),s_{n,0}\right)\right)\\
 & \leq & \lev\left(s_{n,0}\right)
\end{eqnarray*}
But the left hand side and the right hand side are equal and $\suc\left(\minb\left(s_{n,0}\wedge s_{n,1}\right),s_{n,0}\right)\leq s_{n,0}$,
so 
\[
\suc\left(\minb\left(s_{n,0}\wedge s_{n,1}\right),s_{n,0}\right)=s_{n,0}.
\]

By sequential homogeneity, $\suc\left(\minb\left(s_{n,i}\wedge s_{n,i+1}\right),s_{n,i}\right)=s_{n,i}$
for all $i<\delta$ as desired. 
\end{proof}

\subsection*{Constructing a model in $\mathcal{T}$}

By Assumption \ref{ass:NoColoring}, we have a function $\cc:\left[\kappa\right]^{<\omega}\to\theta$
that witnesses the fact that $\kappa\not\to\left(\delta\right)_{\theta}^{<\omega}$
(the letter $\cc$ stands for ``coloring''). Fix $\cc$, and also
a pairing function (a bijection) $\pr:\theta\times\theta\to\theta$
and projections $\pi_{1},\pi_{2}:\theta\to\theta$ (defined so that
$\pi_{1}\left(\pr\left(\alpha,\beta\right)\right)=\alpha$ and $\pi_{2}\left(\pr\left(\alpha,\beta\right)\right)=\beta$).
For us, $0$ is considered to be a limit ordinal. For an ordinal $\alpha$,
let $\Lim\left(\alpha\right)=\left\{ \beta<\alpha\left|\,\beta\mbox{ is a limit}\right.\right\} $.
\begin{defn}
\label{def:P-1-1-1} $\P=\P_{\theta,\kappa}$ is the set of triples
$\p=\left(d,M,E\right)=\left(d_{\p},M_{\p},E_{\p}\right)$ such that:
\begin{enumerate}
\item $M$ is a standard model of $T_{\left\{ \emptyset\right\} }^{\forall}$
and $M=P_{\emptyset}^{M}$ (i.e., $M$ is just a standard tree). Some
notation:

\begin{enumerate}
\item We write $<_{\p}$ instead of $<_{\emptyset}^{M_{\p}}$ etc., or omit
$\p$ when it is clear from the context.
\item Let $\Suc_{\lim}\left(M\right)$ be the set of all $t\in\Suc\left(M\right)$
such that $\lev\left(t\right)-1$ is a limit.
\end{enumerate}
\item $E$ is an equivalence relation refining $E^{\nb}$ (see Notation
\ref{not:neighbors}). Moreover, for levels that are not $\alpha+1$
for limit $\alpha$ it equals $E^{\nb}$. By normality $E$ is equality
on limit elements, so it is interesting only on $\Suc_{\lim}\left(M\right)$.
\item For every $E^{\nb}$ equivalence class $C$, $\left|C/E\right|<\kappa$.
\item $d$ is a function from $\left\{ \eta\in\Suc_{\lim}\left(M\right)^{<\omega}\left|\,\eta\left(0\right)<\cdots<\eta\left(\lg\left(\eta\right)-1\right)\right.\right\} $
to $\theta$.
\item We say that $\p$ is \emph{hard} if there is no increasing sequence
of elements $\bar{s}$ of length $\delta$ from $\Suc_{\lim}\left(M\right)$
such that:

For all $n<\omega$ there is $c_{n}<\theta$ such that for every $i_{0}<\cdots<i_{n-1}<\delta$,
$d\left(s_{i_{0}},\ldots,s_{i_{n-1}}\right)=c_{n}$.

\end{enumerate}
\end{defn}
\begin{example}
\label{exa:P_c-1-1-1} Consider $\left(\kappa,<\right)$ as a standard
tree. Let $\p_{\cc}=\left(\cc\upharpoonright\Suc_{\lim}\left(\kappa\right),\kappa,=\right)\in\P$.
Then $\p_{\cc}$ is hard.\end{example}
\begin{defn}
Let $\p=\left(d_{\p},M_{\p},E_{\p}\right)\in\P$, let $x$ be a variable
and $A\subseteq\Suc_{\lim}\left(M_{\p}\right)$ be a linearly ordered
set.
\begin{enumerate}
\item Say that $p$ is a \emph{$d$-type }over $A$ if $p$ is a consistent
set of equations of the form

$d\left(a_{0},\ldots,a_{n-1},x\right)=\varepsilon$ where $n<\omega$,
$\varepsilon<\theta$ and $a_{0}<\cdots<a_{n-1}\in A$.

\item Consistency here means that $p$ does not contain a subset of the
form 
\[
\left\{ d\left(a_{0},\ldots,a_{n-1},x\right)=\varepsilon,d\left(a_{0},\ldots,a_{n-1},x\right)=\varepsilon'\right\} 
\]
 for $\varepsilon\neq\varepsilon'$.
\item Say that $p$ is complete if for every increasing sequence $\left\langle a_{0},\ldots,a_{n-1}\right\rangle $
from $A$ there is such an equation in $p$.
\item If $B\subseteq A$ then for a $d$-type $p$ over $A$, let 
\[
p\upharpoonright B=\left\{ d\left(a_{0},\ldots,a_{n-1},x\right)=\varepsilon\in p\left|\, a_{0},\ldots,a_{n-1}\in B\right.\right\} .
\]
 
\item For $t\in\Suc_{\lim}\left(M_{\p}\right)$,
\begin{eqnarray*}
\dtp\left(t/A\right) & = & \left\{ d\left(a_{0},\ldots,a_{n-1},x\right)=\varepsilon\right|\\
 &  & \left.a_{0}<\cdots<a_{n-1}\in A,a_{n-1}<t,d_{\p}\left(a_{0},\ldots,a_{n-1},t\right)=\varepsilon\right\} .
\end{eqnarray*}
For an element $t\in\Suc_{\lim}\left(M\right)$, $t\models p$ means
that $t$ satisfies all the equations in $p$ when we replace $d$
by $d_{p}$.
\item Let $S_{d}\left(A\right)$ be the set of all complete $d$-types over
$A$.
\end{enumerate}
\end{defn}
Now we define the function $\q$ from $\P$ to $\P$.
\begin{defn}
\label{def:q-1-1}For $\p=\left(M_{\p},d_{\p},E_{\p}\right)\in\P$,
define $\q=\q\left(\p\right)=\left(M_{\q},d_{\q},E_{\q}\right)\in\P$
by:
\begin{itemize}
\item $M_{\q}$ is the set of pairs $a=\left(\Gamma,\eta\right)=\left(\Gamma_{a},\eta_{a}\right)$
such that:

\begin{enumerate}
\item There is $\alpha<\kappa$ such that $\eta:\alpha\to\Suc_{\lim}\left(M_{\p}\right)$
and $\Gamma:\Lim\left(\alpha\right)\to S_{d}\left(M_{\p}\right)$.
Denote $\lg\left(\Gamma,\eta\right)=\lg\left(\eta\right)=\alpha$.
If $\alpha$ is a successor ordinal, let $l_{\left(\Gamma,\eta\right)}=\eta\left(\alpha-1\right)\in M_{\p}$.
\item For $\beta<\alpha$ limit, $\Gamma\left(\beta\right)\in S_{d}\left(\left\{ \eta\left(\beta'\right)\left|\,\beta'\leq\beta\right.\right\} \right)$.
\item \label{enu:ZeroSatisfies} If $0<\alpha$ then $\eta\left(0\right)\models\Gamma\left(0\right)\upharpoonright\emptyset$.
\item For $\beta'<\beta<\alpha$, $\eta\left(\beta'\right)<_{\p}\eta\left(\beta\right)$
($\eta$ is increasing in $M_{\p}$).
\item If $\beta'<\beta<\alpha$ are limit ordinals then $\Gamma\left(\beta'\right)\subseteq\Gamma\left(\beta\right)$.
\item \label{enu:satisfies} If $\beta'<\beta<\alpha$ and $\beta'$ is
a limit ordinal then $\eta\left(\beta\right)\models\Gamma\left(\beta'\right)$.
\item \label{enu:tightness-1-1} For $\beta<\alpha$, there is no $t<_{\p}\eta\left(\beta\right)$
that satisfies

\begin{enumerate}
\item $t\in\Suc_{\lim}\left(M_{\p}\right)$,
\item $\eta\left(\beta'\right)<_{\p}t$ for all $\beta'<\beta$,
\item $t\models\Gamma\left(0\right)\upharpoonright\emptyset$, and
\item $t\models\Gamma\left(\beta'\right)$ for all limit $\beta'<\beta$.
\end{enumerate}
\item The order on $M_{\q}$ is $\left(\Gamma,\eta\right)<_{\q}\left(\Gamma',\eta'\right)$
iff $\Gamma\triangleleft\Gamma'$ and $\eta\triangleleft\eta'$ (where
$\triangleleft$ means first segment). This defines a standard tree
structure on $M_{\q}$.

It follows that for $a=\left(\Gamma,\eta\right)$, $\lev\left(a\right)=\lg\left(a\right)$.

\end{enumerate}
\item $d_{\q}$ is defined as follows: suppose $a_{0}<_{\q}\cdots<_{\q}a_{n-1}\in\Suc_{\lim}\left(M_{\q}\right)$
and $a_{i}=\left(\Gamma_{i},\eta_{i}\right)$.

Let $t_{i}=l_{a_{i}}=\eta_{i}\left(\lg\left(a_{i}\right)-1\right)$
and $p=\Gamma_{n-1}\left(\lg\left(a_{n-1}\right)-1\right)$. Let $\varepsilon\in\theta$
be the unique color such that $d\left(t_{0},\ldots,t_{n-1},x\right)=\varepsilon\in p$.
Then 
\[
d_{\q}\left(a_{0},\ldots,a_{n-1}\right)=\pr\left(\varepsilon,\cc\left(\lev\left(a_{0}\right),\ldots,\lev\left(a_{n-1}\right)\right)\right).
\]

\item $E_{\q}$ is defined as follows: $\left(\Gamma_{1},\eta_{1}\right)\mathrela{E_{\q}}\left(\Gamma_{2},\eta_{2}\right)$
iff 

\begin{itemize}
\item $\lg\left(\eta_{1}\right)=\lg\left(\eta_{2}\right)$, so equals to
some $\alpha<\kappa$,
\item $\eta_{1}\upharpoonright\beta=\eta_{2}\upharpoonright\beta,\Gamma_{1}\upharpoonright\beta=\Gamma_{2}\upharpoonright\beta$
for all $\beta<\alpha$ (so they are $E^{\nb}$-equivalent),
\item $\Gamma_{1}\left(0\right)\upharpoonright\emptyset=\Gamma_{2}\left(0\right)\upharpoonright\emptyset$,
and
\item If $\alpha=\beta+n$ for $\beta\in\Lim\left(\alpha\right)$ and $n<\omega$
then for all $\alpha_{0}<\alpha_{1}<\cdots<\alpha_{k-1}<\beta$, 
\begin{eqnarray*}
d\left(\eta_{1}\left(\alpha_{0}\right),\ldots,\eta_{1}\left(\alpha_{k-1}\right),\eta_{1}\left(\beta\right),x\right)=\varepsilon\in\Gamma_{1}\left(\beta\right) & \Leftrightarrow\\
d\left(\eta_{2}\left(\alpha_{0}\right),\ldots,\eta_{2}\left(\alpha_{k-1}\right),\eta_{2}\left(\beta\right),x\right)=\varepsilon\in\Gamma_{2}\left(\beta\right)
\end{eqnarray*}
Note that it follows that if $1<n$, and $\left(\Gamma_{1},\eta_{1}\right)\mathrela{E^{\nb}}\left(\Gamma_{2},\eta_{2}\right)$,
then $\Gamma_{1}\left(\beta\right)=\Gamma_{2}\left(\beta\right)$
and $\eta_{1}\left(\beta\right)=\eta_{2}\left(\beta\right)$, so they
are $E$-equivalent.
\end{itemize}
\end{itemize}
\end{defn}
In the next claims we assume that $\p\in\P$ and $\q=\q\left(\p\right)$.
\begin{rem}
\label{rem:LevelVsLength-1-1} $\lev\left(a\right)=\lg\left(a\right)$
for $a\in M_{\q}$ and $a\mathrela{E^{\nb}}b$ iff $\lev\left(a\right)=\lev\left(b\right)$
and $a\upharpoonright\alpha=b\upharpoonright\alpha$ for all $\alpha<\lev\left(a\right)$.\end{rem}
\begin{claim}
$\q\in\P_{\theta,\kappa}$ and moreover it is hard.\end{claim}
\begin{proof}
The fact that $M_{\q}$ is a standard tree is trivial. Also, $E$
refines $E^{\nb}$ by definition.

We must show that the number of $E$-classes inside a given $E^{\nb}$-class
is bounded.

Given a (partial) $d$-type $p$ over $M_{\p}$ and $t\in M_{\p}$,
let $p^{t}$ be the set of equations we get by replacing all appearances
of $t$ by a special letter $*$.

Assume that $A$ is an $E^{\nb}$-class contained in $\Suc_{\lim}\left(M_{\q}\right)$,
and that for every $a\in A$, $\lev\left(a\right)=\alpha+1$ where
$\alpha$ is limit. Assume $a\in A$ and let $B=\left\{ *\right\} \cup\image\left(\eta_{a}\right)\backslash\left\{ l_{a}\right\} $
(since $A$ is an $E^{\nb}$-class, this set does not depend on the
choice of $a$). Consider the map $\varepsilon$ defined by $a\mapsto\Gamma_{a}\left(\alpha\right)^{l_{a}}$.
Then, $a,b\in A$ are $E$ equivalent iff\emph{ }$\varepsilon\left(a\right)=\varepsilon\left(b\right)$\emph{.}
Therefore this map induces an injective map from $A/E$ to this set
of types. The size of this set is at most $2^{\left|B\right|+\theta+\aleph_{0}}$.
But $\left|B\right|=\left|\alpha\right|<\kappa$, and $\theta<\kappa$
by assumption, so $\left|A/E\right|<\kappa$ (as $\kappa$ is a strong
limit).

$\q$ is hard: if $\bar{s}=\left\langle s_{i}\left|\, i<\delta\right.\right\rangle $
is a counterexample then $\left\langle \lev\left(s_{i}\right)\left|\, i<\delta\right.\right\rangle $
would be a homogeneous sub-sequence, contradicting the choice of $\cc$.\end{proof}
\begin{prop}
\label{prop:LevelDecreases-1}$ $
\begin{enumerate}
\item \label{enu:levDec} For all $a\in\Suc\left(M_{\q}\right)$, $\lev_{M_{\q}}\left(a\right)\leq\lev_{M_{\p}}\left(l_{a}\right)$.
\item Assume $t\in\Suc_{\lim}\left(M_{\p}\right)$. Then there is some $a=\left(\Gamma,\eta\right)\in\Suc\left(M_{\q}\right)$
such that $l_{a}=t$.
\end{enumerate}
\end{prop}
\begin{proof}
(1) Let $\lev_{M_{\q}}\left(a\right)=\alpha$. Then $\left\langle \pre_{\p}\left(\eta_{a}\left(\beta\right)\right)\left|\,\beta<\alpha\right.\right\rangle $
is an increasing sequence below $l_{a}$, hence $\alpha\leq\lev_{M_{\p}}\left(l_{a}\right)$.

(2) Let $\Gamma$ be the set of ordinals $\gamma$ for which there
is a sequence $\sequence{\left(\Gamma_{\alpha},\eta_{\alpha}\right)}{\alpha<\gamma}$
such that for every $\alpha<\gamma$:
\begin{itemize}
\item [$\star$]$\left(\Gamma_{\alpha},\eta_{\alpha}\right)\in M_{\q}$;
$\lg\left(\eta_{\alpha}\right)=\alpha$; it is an increasing sequence
in $<_{\q}$; $\eta_{\alpha}\left(\beta\right)<t$ for $\beta<\alpha$
and if $\beta$ is a limit then $\Gamma_{\alpha}\left(\beta\right)=\dtp\left(t/\left\{ \eta_{\alpha}\left(\beta'\right)\left|\,\beta'\leq\beta\right.\right\} \right)$.
\end{itemize}
We try to construct such a sequence $\sequence{\left(\Gamma_{\alpha},\eta_{\alpha}\right)}{\alpha<\gamma}$
as long as we can. By (1), $\lev_{M_{\p}}\left(t\right)+1\notin\Gamma$,
so $\gamma<\kappa$ and $\gamma$ must be a successor ordinal. Let
$\beta=\gamma-1$. 

Define $\eta=\eta_{\beta}\cup\left\{ \left(\beta,t\right)\right\} $,
$\Gamma=\Gamma_{\beta}$ unless $\beta$ is a limit, in which case
let $\Gamma\left(\beta\right)$ be any complete type in $x$ over
$\left\{ \eta\left(\beta'\right)\left|\,\beta'\leq\beta\right.\right\} $
containing $\bigcup\left\{ \Gamma_{\beta}\left(\beta'\right)\left|\,\beta'\in\Lim\left(\beta\right)\right.\right\} \cup\left\{ d\left(x\right)=d_{\p}\left(t\right)\right\} $. 

By construction, $\left(\Gamma,\eta\right)\in M_{\q}$.
\end{proof}
Now we build a model in $\mathcal{T}$ using $\P$:
\begin{defn}
\label{def:the model M_c}
\begin{enumerate}
\item Define $\p_{0}=\p_{\cc}$ (see Example \ref{exa:P_c-1-1-1}), and
for $n<\omega$, let $\p_{n+1}=\q\left(\p_{n}\right)$.
\item Define $P_{n}=M_{\p_{n}}$, $d_{n}=d_{\p_{n}}$ and $E_{n}=E_{\p_{n}}$.
\item Let $M_{\cc}=\bigcup_{n<\omega}P_{n}$ (we assume that the $P_{n}$'s
are mutually disjoint). So $P_{n}^{M_{\cc}}=P_{n}$.
\item $M_{\cc}\models T_{\omega}^{\forall}$ when we interpret the relations
in the language as they are induced from each $P_{n}$ and in addition:
\item Define $G_{n}^{M_{c}}:\Suc\left(P_{n}\right)\to\Suc\left(P_{n+1}\right)$
as follows: let $a\in\Suc\left(P_{n}\right)$ and $a'=\suc\left(\limb\left(a\right),a\right)$.
By Proposition \ref{prop:LevelDecreases-1}, there is an element $\left(\Gamma,\eta\right)_{a}\in\Suc\left(P_{n+1}\right)$
such that $l_{\left(\Gamma,\eta\right)_{a}}=a'$. Choose such an element
for each $a$, and define $G_{n}^{M_{\cc}}\left(a\right)=\left(\Gamma,\eta\right)_{a}$.
\end{enumerate}
\end{defn}
\begin{cor}
$M_{\cc}\in\mathcal{T}$.\end{cor}
\begin{proof}
All the demands of Definition \ref{def:ClassT} are easy. For instance,
Clause (\ref{enu:GDecreasing}) follows from Proposition \ref{prop:LevelDecreases-1}.
Clause (\ref{enu:GInjective}) follows from the fact that if $\left\langle s_{i}\left|\, i<\delta\right.\right\rangle $
is an increasing sequence in $P_{n}$ such that $s_{i}\equiv s_{j}\modp{\omega}$
then $\left\langle \suc\left(\minb\left(s_{i}\right),s_{i}\right)\left|\, i<\delta\right.\right\rangle $
is increasing, so $l_{G_{n}\left(s_{i}\right)}\neq l_{G_{n}\left(s_{j}\right)}$
for $i\neq j$.\end{proof}
\begin{notation}
Again, we do not write the index $\p_{n}$ when it is clear from the
context (for instance we write $d\left(s_{0},\ldots,s_{k}\right)$
instead of $d_{\p_{n}}\left(s_{0},\ldots,s_{k}\right)$).
\end{notation}
The following lemma and corollary will show that starting with any
HNI sequence in $M_{\cc}$, by applying $H$ to it many times, we
must get a fan.
\begin{lem}
\label{lem:KeyClaim-1-1} Assume that $\bar{s}\in\HNind_{\ai}\left(M_{\cc}\right)$
and $\bar{t}=H\left(\bar{s}\right)\in\HNind_{\ai}\left(M_{\cc}\right)$
(see Definition \ref{def:TheHFunction-1-1}) satisfy that for all
$i<\delta$:
\begin{itemize}
\item $\suc\left(\minb\left(s_{i}\wedge s_{i+1}\right),s_{i}\right)=s_{i}$,
and
\item $\suc\left(\minb\left(t_{i}\wedge t_{i+1}\right),t_{i}\right)=t_{i}$.
\end{itemize}
Then, letting $u_{i}=\suc\left(\minb\left(s_{i}\wedge s_{i+1}\right),s_{i+1}\right)$
and $v_{i}=\suc\left(\minb\left(t_{i}\wedge t_{i+1}\right),t_{i+1}\right)$
for $i<\delta$:
\begin{enumerate}
\item \textup{$\left\langle d\left(u_{i}\right)\left|\,1\leq i<\delta\right.\right\rangle $
is constant.}
\item \textup{$d\left(u_{i_{0}},\ldots,u_{i_{n}}\right)=\pi_{1}\left(d\left(v_{i_{0}},\ldots,v_{i_{n-1}}\right)\right)$
for $1\leq i_{0}<\cdots<i_{n}<\delta$ (recall that $\pi_{1}$ is
defined by $\pi_{1}\left(\pr\left(i,j\right)\right)=i$).}
\end{enumerate}
\end{lem}
\begin{proof}
(1) By definition, $t_{i}=G\left(u_{i}\right)$. Denote $t_{i}=\left(\Gamma_{i},\eta_{i}\right)$.
As $\left\langle t_{i}\wedge t_{i+1}\left|\, i<\delta\right.\right\rangle $
is an increasing sequence (because $\bar{t}\in\HNind_{\ai}\left(M_{\cc}\right)$),
$0<\lev\left(t_{1}\wedge t_{2}\right)$. Let $p=\Gamma_{t_{1}\wedge t_{2}}\left(0\right)\upharpoonright\emptyset$.
Then $p=\Gamma_{i}\left(0\right)\upharpoonright\emptyset$ for all
$1\leq i$ (it may be that $t_{1}\wedge t_{0}=\emptyset$ and in this
case we have no information on $t_{0}$). Assume that $p=\left\{ d\left(x\right)=\varepsilon\right\} $
for some $\varepsilon<\theta$. Then, by Definition \ref{def:q-1-1},
Clauses (\ref{enu:ZeroSatisfies}) and (\ref{enu:satisfies}), $d\left(\eta_{i}\left(\beta\right)\right)=\varepsilon$
for all $1\leq i<\delta$ and $\beta<\lg\left(\eta_{i}\right)$. As
$u_{i}=l_{t_{i}}$ we are done.

(2) Denote $v_{i}=\left(\Gamma_{i}',\eta_{i}'\right)$. By our assumptions
on $\bar{t}$, $t_{i}\mathrela{E^{\nb}}v_{i}$ hence if $\bar{t}$
is increasing then $\bar{v}=\bar{t}$. Assume that it is not increasing.
Then $t_{i}\wedge t_{i+1}<t_{i}$ so $\minb\left(t_{i}\wedge t_{i+1}\right)=t_{i}\wedge t_{i+1}$.
Let $\alpha_{i}=\beta_{i}+1=\lev\left(t_{i}\right)=\lg\left(\eta_{i}'\right)$,
then $\beta_{i}$ is a limit ordinal and $t_{i}\upharpoonright\beta_{i}=v_{i}\upharpoonright\beta_{i}$.
So for $1\leq i$, $\Gamma'_{i}\left(0\right)\upharpoonright\emptyset=\Gamma_{i}\left(0\right)\upharpoonright\emptyset=p$
and $\Gamma_{i}'\upharpoonright\beta_{i}=\Gamma_{i}\upharpoonright\beta_{i}$.

Note that for $1\leq i$, $l_{t_{i}}$ and $l_{v_{i}}$ are both below
$u_{i+1}=l_{t_{i+1}}$ (as $v_{i}\leq t_{i+1}$ and ${l_{t_{i}}=u_{i}<u_{i+1}}$),
that they both satisfy $p$ and that they both satisfy the equations
in $\Gamma\left(\beta\right)$ for each limit $\beta<\beta_{i}$,
so if for instance $l_{t_{i}}<l_{v_{i}}$, we will have a contradiction
to Definition \ref{def:q-1-1}, Clause (\ref{enu:tightness-1-1}). 

So, in any case (whether or not $\bar{t}$ is increasing), we have
$l_{v_{i}}=l_{t_{i}}=u_{i}$.

By choice of $\bar{v}$ and the assumptions on $\bar{t}$, $\bar{v}$
is increasing so $d$ is defined on finite subsets of it. 

Assume $1\leq i_{0}<\cdots<i_{n}<\delta$. Then for every $\sigma<\theta$,
by the choice of $d$ in Definition \ref{def:q-1-1}:
\begin{itemize}
\item [$\boxtimes$]$\pi_{1}\left(d\left(v_{i_{0}},\ldots,v_{i_{n-1}}\right)\right)=\sigma$
iff 
\item [$\boxtimes$]$d\left(l_{v_{i_{0}}},\ldots,l_{v_{i_{n-1}}},x\right)=\sigma\in\Gamma'_{i_{n-1}}\left(\beta_{i_{n-1}}\right)$
iff
\item [$\boxtimes$]$d\left(l_{v_{i_{0}}},\ldots,l_{v_{i_{n-1}}},x\right)=\sigma\in\Gamma'_{i_{n}}\left(\beta_{i_{n-1}}\right)$
(because $\Gamma'_{i_{n}}\upharpoonright\alpha_{i_{n-1}}=\Gamma'_{i_{n-1}}\upharpoonright\alpha_{i_{n-1}}$)
iff
\item [$\boxtimes$]$d\left(l_{v_{i_{0}}},\ldots,l_{v_{i_{n-1}}},l_{v_{i_{n}}}\right)=\sigma$
(this follows from Clause (\ref{enu:satisfies}) of Definition \ref{def:q-1-1})
iff
\item [$\boxtimes$]$d\left(u_{i_{0}},\ldots,u_{i_{n}}\right)=\sigma$ (because
$l_{v_{i}}=u_{i}$).
\end{itemize}
\end{proof}
\begin{cor}
\label{cor:mustBeFan-1-1}If $\bar{s}\in\HNind_{\ai}\left(M_{\cc}\right)$
then there must be some $n<\omega$ such that $H^{\left(n\right)}\left(\bar{s}\right)\in\HNind_{f}\left(M_{\cc}\right)$
(see Definition \ref{def:TheHFunction-1-1}).\end{cor}
\begin{proof}
If not, by Corollary \ref{cor:TheHFunction-1-1}, for all $n<\omega$,
$H^{\left(n\right)}\left(\bar{s}\right)\in\HNind_{\ai}\left(M_{\cc}\right)$.
Moreover, there exists some $K<\omega$ such that for all $K\leq n$,
if $\bar{t}=H^{\left(n\right)}\left(\bar{s}\right)$ then $\suc\left(\minb\left(t_{i}\wedge t_{i+1}\right),t_{i}\right)=t_{i}$.
Without loss, $K=0$ (i.e., this is true also for $\bar{s}$).
\begin{claim*}
If $\bar{s}$ is such a sequence then for all $n<\omega$, $d\left(u_{i_{0}},\ldots,u_{i_{n-1}}\right)$
is constant for all $1\leq i_{0}<\cdots<i_{n-1}<\delta$ where $u_{i}=\suc\left(\minb\left(s_{i}\wedge s_{i+1}\right),s_{i+1}\right)$
for $i<\delta$.\end{claim*}
\begin{proof}
(of claim) Prove by induction on $n$ using Lemma  \ref{lem:KeyClaim-1-1}.
\end{proof}
But this claim contradicts the fact that for all $k<\omega$, $\p_{k}$
is hard.
\end{proof}
The next lemma and corollaries are the main conclusion of this section:
\begin{lem}
\label{lem:MustBeGood-1-1} If $\bar{s}\in\HNind_{\ai}\left(M_{\cc}\right)$
and $\bar{t}=H\left(\bar{s}\right)\in\HNind_{f}\left(M_{\cc}\right)$
then $\neg\left(v_{i}\mathrela Ev_{j}\right)$ for $i<j<\delta$ where
$v_{i}=\suc\left(\minb\left(t_{i+1}\wedge t_{i}\right),t_{i}\right)$.\end{lem}
\begin{proof}
Let $t=t_{0}\wedge t_{1}$, so $t=t_{i}\wedge t_{j}$ for all $i<j<\delta$.
Let $u_{i}=\suc\left(t,t_{i}\right)$. As $t_{i}\neq t_{j}$ for $i<j<\delta$,
$u_{i}\neq u_{j}$. In addition
\[
l_{u_{i}}\leq l_{t_{i}}=\suc\left(\minb\left(s_{i}\wedge s_{i+1}\right),s_{i+1}\right)\leq s_{i+1}\wedge s_{i+2}
\]
and $\left\langle s_{i}\wedge s_{i+1}\left|\, i<\delta\right.\right\rangle $
is increasing so $l_{u_{i}}$ and $l_{u_{j}}$ are comparable. 

First assume that $\alpha=\lev\left(t\right)>0$. Then $\Gamma_{t}\left(0\right)=\Gamma_{t_{i}}\left(0\right)$
for $i<\delta$. For all $i<j<\delta$, $l_{u_{i}}\models\Gamma_{u_{j}}\left(0\right)\upharpoonright\emptyset$,
$l_{u_{i}}$ is greater than $\eta_{u_{j}}\left(\beta\right)=\eta_{t}\left(\beta\right)$
for all $\beta<\alpha$ and $l_{u_{i}}\models\Gamma_{u_{j}}\left(\beta\right)=\Gamma_{t}\left(\beta\right)$
for all limit $\beta<\alpha$. So by Definition \ref{def:q-1-1},
Clause (\ref{enu:tightness-1-1}), $l_{u_{i}}=l_{u_{j}}$, so $\eta_{u_{i}}=\eta_{u_{j}}$
for all $i<j<\delta$.

But since $u_{i}\neq u_{j}$, it necessarily follows that $\Gamma_{u_{i}}\neq\Gamma_{u_{j}}$.
If $\alpha=\beta+1$ for some $\beta$, then by definition of the
function $\q$, $\Gamma_{u_{i}}=\Gamma_{u_{i}}\upharpoonright\alpha=\Gamma_{t}$
(because $\Gamma$ was defined only for limit ordinals). So necessarily
$\alpha$ is a limit, and it follows that $\limb\left(t\right)=t$
so $v_{i}=u_{i}$. Now it is clear that $\Gamma_{v_{i}}\left(\alpha\right)\neq\Gamma_{v_{j}}\left(\alpha\right)$
and by definition of $E$, $\neg\left(v_{i}\mathrela Ev_{j}\right)$
for all $i<j<\delta$.

If $\alpha=0$, then as before $v_{i}=u_{i}$ (because $\limb\left(t\right)=t$).
We cannot use the same argument (because $\Gamma_{t}\left(0\right)$
is not defined), so we take care of each pair $i<j<\delta$ separately.
If $\Gamma_{v_{i}}\left(0\right)\upharpoonright\emptyset=\Gamma_{v_{j}}\left(0\right)\upharpoonright\emptyset$
then the argument above will work and $\neg\left(v_{i}\mathrela Ev_{j}\right)$.
If $\Gamma_{v_{i}}\left(0\right)\upharpoonright\emptyset\neq\Gamma_{v_{j}}\left(0\right)\upharpoonright\emptyset$,
then $\neg\left(v_{i}\mathrela Ev_{j}\right)$ follows directly from
the definition.
\end{proof}
Finally we have
\begin{cor}
If $\bar{s}\in\HNind_{\ai}\left(M_{\cc}\right)$, then there is some
$\bar{v}\in\HNind_{f}\left(M_{\cc}\right)$ such that $v_{i}=\suc\left(\limb\left(v_{i}\right),v_{i}\right)$,
$v_{i}\mathrela{E^{\nb}}v_{j}$ but $\neg\left(v_{i}\mathrela Ev_{j}\right)$
for $i<j<\delta$.\end{cor}
\begin{proof}
By Corollary \ref{cor:mustBeFan-1-1}, there is some minimal $n<\omega$
such that $\bar{t}=H^{\left(n+1\right)}\left(\bar{s}\right)\in\HNind_{f}\left(M_{\cc}\right)$.
Let $v_{i}=\suc\left(\minb\left(t_{i+1}\wedge t_{i}\right),t_{i}\right)$
for $i<\delta$. By Lemma \ref{lem:MustBeGood-1-1}, we have that
$v_{i}\mathrela{E^{\nb}}v_{j}$ but $\neg\left(v_{i}\mathrela Ev_{j}\right)$
for $i<j<\delta$ (in particular $v_{i}\neq v_{j}$). So necessarily
$t=t_{i}\wedge t_{j}$ is a limit and $v_{i}=\suc\left(t,v_{i}\right)$. \end{proof}
\begin{cor}
\label{cor:MainCorIacc}If there is some $\bar{s}\in\ind\left(M_{\cc}\right)$
such that $s_{i}\in P_{0}^{M_{\cc}}$ for all $i<\delta$, then there
is some $\bar{v}\in\ind_{f}\left(M_{\cc}\right)$ such that $v_{i}\in\Suc_{\lim}\left(M_{\cc}\right)$,
$v_{i}\mathrela{E^{\nb}}v_{j}$ but $\neg\left(v_{i}\mathrela Ev_{j}\right)$
for $i<j<\delta$.\end{cor}
\begin{proof}
Since $P_{0}=\kappa$, any sequence $\bar{s}$ in $\ind\left(M_{\cc}\right)$
in $P_{0}$ must be increasing. So by the last corollary there is
some $\bar{v}\in\HNind_{f}\left(M_{\cc}\right)$ like there. But then
by sparseness (see Definition \ref{def:NI-1-1}) there is some $n<\omega$
such that $\left\langle v_{ni}\left|\, i<\delta\right.\right\rangle $
is indiscernible. \end{proof}
\begin{rem}
In this section it becomes clear why we needed to use discrete trees
and not dense ones (as in \cite{KaSh946}). In Corollary \ref{cor:MainCorIacc},
we started with an increasing sequence in $P_{0}=\kappa$, and then
applied a definable map on it, to get a new HNI sequence $\bar{s}$,
but this sequence might be almost increasing and not increasing (i.e.,
in $\ind_{\ai}$). Since we wanted the coloring function $d$ to be
defined on increasing sequences, we needed again to get an increasing
sequence, so this is done by taking $s_{i}\wedge s_{i+1}$. This sequence
is increasing, but in order for the coloring $d$ to affect the coloring
of the original sequence (as in Lemma \ref{lem:KeyClaim-1-1}), we
need this definable map to give us a successor of $s_{i}\wedge s_{i+1}$.
Trial and error has shown that adding the function ``successor to
the meet'' instead of just successor results in losing AP, so we
needed the successor function. The predecessor function is not necessary
(in existentially closed models, if $x>\limb_{\eta}\left(x\right)$,
$x$ has a predecessor), but there is no price to adding it, and it
simplifies the theory a bit.
\end{rem}

\section{\label{sec:The-main-theorem}Proof of the main theorem}

In this section we prove Main Theorem \ref{mainthmA}. 

We start with the easy direction. 
\begin{prop}
\label{prop:Easy direction}Let $\kappa,\theta$ be cardinals and
$\delta\geq\omega$ a limit ordinal. If $\kappa\to\left(\delta\right)_{\theta}^{<\omega}$
then for every $n\leq\omega$ and every theory $T$ of cardinality
$\left|T\right|\leq\theta$, $\kappa\to\left(\delta\right)_{T,n}$.\end{prop}
\begin{proof}
For convenience, let $\bar{x}_{i}$ for $i<\omega$ be disjoint $n$-tuples
of variables and let $L\left(T\right)$ be the set of formulas in
$T$ in $\set{\bar{x}_{i}}{i<\omega}$. 

Let $\left\langle \bar{a}_{i}\left|\, i<\kappa\right.\right\rangle $
be a sequence of $n$-tuples in a model $M\models T$. Define $c:\left[\kappa\right]^{<\omega}\to L\left(T\right)\cup\left\{ 0\right\} $
as follows:

Given an increasing sequence $\eta\in\kappa^{<\omega}$, if $\lg\left(\eta\right)$
is odd, then $c\left(\eta\right)=0$. If not, assume it is $2k$ and
that $\eta=\left\langle \alpha_{i}\left|\, i<2k\right.\right\rangle $.
If $\bar{a}_{\alpha_{0}}\cdots\bar{a}_{\alpha_{k-1}}\equiv\bar{a}_{\alpha_{k}}\cdots\bar{a}_{\alpha_{2k-1}}$
then $c\left(\eta\right)=0$. If not there is a formula $\varphi\left(\bar{x}_{0},\ldots,\bar{x}_{k-1}\right)$
such that $M\models\varphi\left(\bar{a}_{\alpha_{0}},\ldots,\bar{a}_{\alpha_{k-1}}\right)\land\neg\varphi\left(\bar{a}_{\alpha_{k}},\ldots,\bar{a}_{\alpha_{2k-1}}\right)$,
so choose such a $\varphi$ and define $c\left(\eta\right)=\varphi$.
By assumption there is a sub-sequence $\left\langle \bar{a}_{\alpha_{i}}\left|\, i<\delta\right.\right\rangle $
on which $c$ is homogeneous. Without loss, assume that $\alpha_{i}=i$
for $i<\delta$.

It follows that $\left\langle \bar{a}_{i}\left|\, i<\delta\right.\right\rangle $
is an indiscernible sequence:

Suppose there are some $i_{0}<i_{1}<\cdots<i_{2k-1}<\delta$ such
that $\bar{a}_{i_{0}}\cdots\bar{a}_{i_{k-1}}\not\equiv\bar{a}_{i_{k}}\cdots\bar{a}_{i_{2k-1}}$.
Since $\delta$ is limit there are some ordinals $i_{2k},\ldots,i_{3k-1}$
such that $i_{2k-1}<i_{2k}<\cdots<i_{3k-1}<\delta$. 

Since $c$ is homogeneous, there is a formula $\varphi$ such that
$c\left(\left\langle i_{k},\ldots,i_{3k-1}\right\rangle \right)=c\left(\left\langle i_{0},\ldots,i_{2k-1}\right\rangle \right)=\varphi$,
meaning that

\[
M\models\varphi\left(\bar{a}_{i_{0}},\ldots,\bar{a}_{i_{k-1}}\right)\land\neg\varphi\left(\bar{a}_{i_{k}},\ldots,\bar{a}_{i_{2k-1}}\right)
\]
 and 

\[
M\models\varphi\left(\bar{a}_{i_{k}},\ldots,\bar{a}_{i_{2k-1}}\right)\land\neg\varphi\left(\bar{a}_{i_{2k}},\ldots,\bar{a}_{i_{3k-1}}\right)
\]
 --- a contradiction.

Now let $i_{0}<\cdots<i_{k-1}<\delta$ be any increasing sequence.
Let $j<\delta$ be greater than $i_{k-1}$. Then $\bar{a}_{i_{0}}\cdots\bar{a}_{i_{k-1}}\equiv\bar{a}_{j}\cdots\bar{a}_{j+k-1}\equiv\bar{a}_{0}\cdots\bar{a}_{k-1}$
and we are done.
\end{proof}
From now on let $\Ss=2^{<\omega}$.

As in Notation \ref{not: what is indiscernible, modulo}, when we
say indiscernible, we mean indiscernible for quantifier free formulas. 

The proof uses the following construction:
\begin{construction}
\label{con:construction A}Assume $S'\subseteq\Ss$ is such that $\nu\in S'\Rightarrow\nu\upharpoonright k\in S'$
for every $k\leq\lg\left(\nu\right)$. Assume $N\models T_{S'}^{\forall,\theta}$
and that for every $\nu\in S'$, if $\nu\concat\left\langle \varepsilon\right\rangle \notin S'$
for $\varepsilon\in\left\{ 0,1\right\} $, we have a model $M_{\nu}^{\varepsilon}\models T_{\Ss}^{\forall,\theta}$.
We may assume all models are disjoint. We build a model $M\models T_{\Ss}^{\forall,\theta}$
such that $M\upharpoonright L_{S'}^{\theta}\supseteq N$ and: for
every $\nu\in S'$ and $\varepsilon\in\left\{ 0,1\right\} $ such
that $\nu\concat\left\langle \varepsilon\right\rangle \notin S'$
and for every $\eta\in\Ss$, $P_{\nu\concat\left\langle \varepsilon\right\rangle \concat\eta}^{M}=P_{\eta}^{M_{\nu}^{\varepsilon}}$.
In general, for every symbol $R_{\eta}$ from $L_{\Ss}^{\theta}$,
let $R_{\nu\concat\left\langle \varepsilon\right\rangle \concat\eta}^{M}=R_{\eta}^{M_{\nu}^{\varepsilon}}$.
For instance, $e_{\nu\concat\left\langle \varepsilon\right\rangle \concat\eta,i}^{M}=e_{\eta,i}^{M_{\nu}^{\varepsilon}}$
for $i<\theta$ and $G_{\nu\concat\left\langle \varepsilon\right\rangle \concat\eta_{1},\nu\concat\left\langle \varepsilon\right\rangle \concat\eta_{2}}^{M}=G_{\eta_{1},\eta_{2}}^{M_{\nu}^{\varepsilon}}$
for $\eta_{1}<_{\suc}\eta_{2}$.

The last thing that remains to be defined is $G_{\nu,\nu\concat\left\langle \varepsilon\right\rangle }^{M}$.
After we have defined it, $M$ is a model. Moreover, for every tuple
$\bar{a}\in M_{\nu}^{\varepsilon}$ and for every quantifier free
formula $\varphi$ from $L_{\Ss}^{\theta}$, there is a formula $\varphi'$
generated by concatenating $\nu\concat\left\langle \varepsilon\right\rangle $
to every symbol appearing in $\varphi$ such that $M_{\nu}^{\varepsilon}\models\varphi\left(\bar{a}\right)$
iff $M\models\varphi'\left(\bar{a}\right)$. In particular, if $I\subseteq M_{\nu}^{\varepsilon}$
is an indiscernible sequence in $M$, it is also such in $M_{\nu}^{\varepsilon}$.
\end{construction}
Main Theorem \ref{mainthmA} follows immediately from Proposition
\ref{prop:Easy direction} and:
\begin{thm}
\label{thm:Main}Let $\Ss=2^{<\omega}$. For any cardinals $\theta$,
$\kappa$ and a limit ordinal $\delta\geq\omega$, $\kappa\to\left(\delta\right)_{T_{\Ss}^{\theta},1}$
iff $\kappa\to\left(\delta\right)_{\theta}^{<\omega}$.\end{thm}
\begin{proof}
We shall prove the following: for every cardinal $\kappa$ and limit
ordinal $\delta\geq\omega$ such that $\kappa\not\to\left(\delta\right)_{\theta}^{<\omega}$,
there is a model $M\models T_{\Ss}^{\forall,\theta}$ and a set $A\subseteq P_{\left\langle \right\rangle }^{M}$
of size $\left|A\right|\geq\kappa$ with no non-constant indiscernible
sequence in $A^{\delta}$. That will suffice (because $M$ can be
extended to a model of $T_{\Ss}^{\theta}$).

The proof is by induction on $\kappa$. Note that if $\kappa\not\to\left(\delta\right)_{\theta}^{<\omega}$
then also $\lambda\not\to\left(\delta\right)_{\theta}^{<\omega}$
for $\lambda<\kappa$. The case analysis for some of the cases is
very similar to the one done in \cite{KaSh946}, but we repeat it
for completeness.
\begin{casenv}
\item \label{cas:.kappa=00003Dtheta} $\kappa\leq\theta$. Let $M\models T_{S}^{\forall,\theta}$
be any model and $A=\left\{ e_{\left\langle \right\rangle ,i}^{M}\left|\, i<\theta\right.\right\} $.
\item $\kappa$ is singular. Assume that $\kappa=\bigcup\left\{ \lambda_{i}\left|\, i<\sigma\right.\right\} $
where $\sigma<\kappa$ and $\lambda_{i}<\kappa$ for all $i<\sigma$.

Assume that $N_{0},A_{0}$ are the model and set given by the induction
hypothesis for $\sigma$. For all $i<\sigma$, let $M_{i},B_{i}$
the models and sets guaranteed by the induction hypothesis for $\lambda_{i}$.
Let $N_{1}$ be a model of $T_{S}^{\forall,\theta}$ containing $M_{i}$
as substructures for all $i<\sigma$ (it exists by JEP) and $A_{1}=\bigcup\left\{ B_{i}\left|\, i<\sigma\right.\right\} $.

Assume that $\left\{ a_{i}\left|\, i<\sigma\right.\right\} \subseteq A_{0}$
and that $\left\{ b_{j}\left|\,\bigcup\left\{ \lambda_{l}\left|\, l<i\right.\right\} \leq j<\lambda_{i}\right.\right\} \subseteq B_{i}$
are enumerations witnessing that $\left|A_{0}\right|\geq\sigma$,
$\left|B_{i}\right|\geq\lambda_{i}\backslash\bigcup\left\{ \lambda_{l}\left|\, l<i\right.\right\} $.

Let $M'\models T_{\left\{ \left\langle \right\rangle \right\} }^{\forall}$
be the standard model (see Definition \ref{def:standard model}) with
$P_{\left\langle \right\rangle }^{M'}=\kappa$ and $\mathordi{<_{\left\langle \right\rangle }^{M'}}=\mathordi{\in}$.

Let $N\models T_{\left\{ \left\langle \right\rangle \right\} }^{\forall,\theta}$
be a model such that $N\upharpoonright L_{\left\{ \left\langle \right\rangle \right\} }\supseteq M'$.
Use Construction \ref{con:construction A} to build a model $M\models T_{\Ss}^{\forall,\theta}$
with $M_{\left\langle \right\rangle }^{0}=N_{0}$ and $M_{\left\langle \right\rangle }^{1}=N_{1}$,
and define the functions $G_{\left\langle \right\rangle ,\left\langle 0\right\rangle }^{M}$
and $G_{\left\langle \right\rangle ,\left\langle 1\right\rangle }^{M}$
as follows: for a limit $\alpha<\kappa$ and $0<n<\omega$, define
$G_{\left\langle \right\rangle ,\left\langle 0\right\rangle }^{M}\left(\alpha+n\right)=a_{\min\left\{ j<\sigma\left|\,\alpha<\lambda_{j}\right.\right\} }$
and $G_{\left\langle \right\rangle ,\left\langle 1\right\rangle }^{M}\left(\alpha+n\right)=b_{\alpha}$.

Let $A=\kappa=P_{\left\langle 0\right\rangle }^{M'}$. Assume that
$\bar{s}=\left\langle s_{i}\left|\, i<\delta\right.\right\rangle $
is an indiscernible sequence contained in $A$.

Obviously it cannot be that $s_{1}<s_{0}$. Assume that $s_{0}<s_{1}$.
There are limit ordinals $\alpha_{i}$ and natural number $n_{i}$
such that $s_{i}=\alpha_{i}+n_{i}$, i.e., $s_{i}\equiv n_{i}\modp{\omega}$.
By indiscernibility, $n_{i}$ is constant, and denote it by $n$.
So $\left\langle \suc\left(s_{2i},s_{2i+1}\right)=s_{2i}+1\left|\, i<\delta\right.\right\rangle $
is an indiscernible sequence of successor ordinals.

$\left\langle G_{\left\langle \right\rangle ,\left\langle 0\right\rangle }\left(s_{2i}+1\right)\left|\, i<\delta\right.\right\rangle $
must be constant by the choice of $A_{0}$, and assume it is $a_{i_{0}}$
for $i_{0}<\sigma$. It follows that $\alpha_{2i}\in\lambda_{i_{0}}\backslash\bigcup\left\{ \lambda_{l}\left|\, l<i_{0}\right.\right\} $.
This means that $G_{\left\langle \right\rangle ,\left\langle 1\right\rangle }\left(s_{2i}+1\right)=b_{\alpha_{2i}}\subseteq B_{i_{0}}$
for all $i<\sigma$, and so $\alpha_{2i}$ must be constant. This
means that $\left\langle s_{2i}\left|\, i<\delta\right.\right\rangle $
is constant so also $\bar{s}$. 

\item \label{cas:kappaRegNotSA} $\kappa$ is regular but not strongly inaccessible.
Then there is some $\lambda<\kappa$ such that $2^{\lambda}\geq\kappa$.

Let $M_{0}\models T_{S}^{\forall,\theta}$ and $A_{0}\subseteq P_{\left\langle \right\rangle }^{M_{0}}$
satisfy the induction hypothesis for $\lambda$. Assume that $A_{0}\supseteq\left\{ a_{i}\left|\, i\leq\lambda\right.\right\} $
where $a_{i}\neq a_{j}$ for $i\neq j$.

Let $M'\models T_{\left\{ \left\langle \right\rangle \right\} }^{\forall}$
be a standard model such that $P_{\left\langle \right\rangle }^{M'}=2^{\leq\lambda}$
ordered by first segment.

Let $N\models T_{\left\{ \left\langle \right\rangle \right\} }^{\forall,\theta}$
be any model such that $N\upharpoonright L_{\left\{ \left\langle \right\rangle \right\} }\supseteq M'$.
We use Construction \ref{con:construction A} to build a model $M\models T_{S}^{\forall}$
using $N$ and $M_{\left\langle \right\rangle }^{0}=M_{\left\langle \right\rangle }^{1}=M_{0}$.
We need to define the functions $G_{\left\langle \right\rangle ,\left\langle 0\right\rangle }$
and $G_{\left\langle \right\rangle ,\left\langle 1\right\rangle }$:

For $f\in P_{\left\langle \right\rangle }^{M'}$ such that $\lg\left(f\right)=\alpha+n$
for some limit $\alpha$ and $n<\omega$, define $G_{\left\langle \right\rangle ,\left\langle 0\right\rangle }^{M'}\left(f\right)=a_{\alpha}$.
There are no further limitations on the functions $G_{\left\langle \right\rangle ,\left\langle 0\right\rangle }^{M'}$
and $G_{\left\langle \right\rangle ,\left\langle 1\right\rangle }^{M'}$
as long as they are regressive.

Let $A=2^{\leq\lambda}=P_{\left\langle \right\rangle }^{M'}$. Assume
for contradiction that $\left\langle s_{i}\left|\, i<\delta\right.\right\rangle $
is a non-constant indiscernible sequence contained in $A$.

It cannot be that $s_{1}<s_{0}$, because by indiscernibility, we
would have an infinite decreasing sequence.

It cannot be that $s_{0}<s_{1}$:

In that case, $\left\langle s_{i}\left|\, i<\delta\right.\right\rangle $
is increasing. For all $i<\delta$, let $t_{i}=\suc\left(s_{2i},s_{2i+1}\right)$.
The sequence $\left\langle t_{i}\left|\, i<\delta\right.\right\rangle $
is an indiscernible sequence contained in $\Suc\left(P_{\left\langle \right\rangle }^{M}\right)$
and so $t_{i}\equiv n\modp{\omega}$ for some constant $n<\omega$.
Hence $\left\langle \lg\left(t_{i}\right)-n\left|\, i<\delta\right.\right\rangle $
is increasing and $\left\langle G_{\left\langle \right\rangle ,\left\langle 0\right\rangle }\left(t_{i}\right)=a_{\lg\left(t_{i}\right)-n}\left|\, i<\delta\right.\right\rangle $
is a non-constant indiscernible sequence contained in $A_{0}$ ---
a contradiction.

Denote $r_{i}=s_{0}\wedge s_{i+1}$ for $i<\delta$. This is an indiscernible
sequence, and by the same arguments, it cannot decrease or increase.
But since $r_{i}<s_{0}$, it follows that $r_{i}$ is constant.

Assume that $s_{0}\wedge s_{1}<s_{1}\wedge s_{2}$, then $s_{1}\wedge s_{2}<s_{2}\wedge s_{3}$
and so $s_{2i}\wedge s_{2i+1}<s_{2\left(i+1\right)}\wedge s_{2\left(i+1\right)+1}$
for all $i<\delta$, and again --- $\left\langle s_{2i}\wedge s_{2i+1}\right\rangle $
is an increasing indiscernible sequence --- we reach a contradiction.

Similarly, it cannot be that $s_{0}\wedge s_{1}>s_{1}\wedge s_{2}$.
As both sides are smaller or equal to $s_{1}$, it must be that 
\[
s_{0}\wedge s_{2}=s_{0}\wedge s_{1}=s_{1}\wedge s_{2}.
\]
 But this is a contradiction (because if $\alpha=\lg\left(s_{0}\wedge s_{1}\right)$
then $\left|\left\{ s_{0}\left(\alpha\right),s_{1}\left(\alpha\right),s_{2}\left(\alpha\right)\right\} \right|=3$,
but the range of these functions is $\left\{ 0,1\right\} $).

\item $\kappa$ is strongly inaccessible.

Assume that $M_{\lambda},A_{\lambda}$ are the models and sets given
by the induction hypothesis for $\lambda<\kappa$. We may assume they
are disjoint. Let $N$ be a model of $T_{\Ss}^{\forall,\theta}$ containing
$M_{\lambda}$ for $\lambda<\kappa$ ($N$ exists by JEP), and let
$A=\bigcup\left\{ A_{\lambda}\left|\,\lambda<\kappa\right.\right\} \subseteq N$.
Recall that we have a function $\cc:\left[\kappa\right]^{<\omega}\to\theta$
that witnesses the fact that $\kappa\not\to\left(\delta\right)_{\theta}^{<\omega}$,
and that in Definition \ref{def:the model M_c} we defined a model
$M_{\cc}$ of $T_{\omega}^{\forall}$. Let $N_{\cc}\models T_{\omega}^{\forall,\theta}$
be a model such that $N_{\cc}\upharpoonright L_{\omega}\supseteq M_{\cc}$.
Let $S'=1^{<\omega}$ (finite sequences of zeros). We may think of
$N_{\cc}$ as a model of $T_{S'}^{\forall,\theta}$. Denote $0_{n}=\left\langle 0,\ldots,0\right\rangle $
where $\lg\left(0_{n}\right)=n$.

We use Construction \ref{con:construction A} and $S'$ to build a
model $M$ of $T_{\Ss}^{\forall,\theta}$:
\begin{itemize}
\item For all $n<\omega$, let $M_{0_{n}}^{1}=N$.
\item Define $G_{0_{n},0_{n}\concat\left\langle 1\right\rangle }^{M}$ as
follows:

\begin{itemize}
\item Recall that $P_{0_{n}}^{M}\supseteq P_{n}^{M_{\cc}}=M_{\p_{n}}$.
Assume that $B\subseteq\Suc_{\lim}\left(P_{n}^{M_{\cc}}\right)$ is
an $E^{\nb}$ class. By definition, $\left|B/E_{\p_{n}}\right|<\kappa$.
\item Choose some enumeration of the classes $\left\{ c_{i}\left|\, i<\left|B/E_{\p_{n}}\right|\right.\right\} $,
and an enumeration $A_{\left|B/E_{\p_{n}}\right|}\supseteq\left\{ a_{i}\left|\, i<\left|B/E_{\p_{n}}\right|\right.\right\} $
of pairwise distinct elements. Now, $G_{0_{n},0_{n}\concat\left\langle 1\right\rangle }^{M}$
maps every class $c_{i}$ (i.e., every element in $c_{i}$) to $a_{i}$.
It is easy to see that if $a\mathrela{E^{\nb}}b$ are distinct in
$\Suc_{\lim}\left(P_{n}^{M_{\cc}}\right)$, then $a$ and $b$ are
not $\sim^{M_{\cc}}$-equivalent (see Definition \ref{def:eq-relation}).
This means that $G_{0_{n},0_{n}\concat\left\langle 1\right\rangle }^{M}$
is well defined. Outside of $P_{n}^{M_{c}}$, define $G_{0_{n},0_{n}\concat\left\langle 1\right\rangle }^{M}$
arbitrarily as long as it is regressive.
\end{itemize}
\end{itemize}

Let $A=\Suc_{\lim}\left(P_{\left\langle \right\rangle }^{M_{\cc}}\right)$,
i.e., $A=\Suc_{\lim}\left(\kappa\right)$. Assume for contradiction
that $A$ contains a $\delta$-indiscernible sequence.

By Corollary \ref{cor:MainCorIacc}, there is $n<\omega$ and an indiscernible
sequence $\bar{v}$ in $\Suc_{\lim}\left(P_{0_{n}}^{M}\right)$ such
that for $i<j<\delta$, $v_{i}\mathrela{E^{\nb}}v_{j}$ but $\neg\left(v_{i}\mathrela Ev_{j}\right)$.
So $\left\langle G_{0_{n},0_{n}\concat\left\langle 1\right\rangle }^{M}\left(v_{i}\right)\left|\, i<\delta\right.\right\rangle $
is a non-constant indiscernible sequence in $A_{\left|\left[v_{0}\right]_{E^{\nb}}/E_{\p_{n}}\right|}$
--- a contradiction.

\end{casenv}
\end{proof}
\begin{rem}
Why in the definition of $\q$ (Definition \ref{def:q-1-1}) we demanded
that the image of $\eta$ is in $\Suc_{\lim}$ and that $\Gamma$
is defined only in limit levels? Had we given $\Gamma$ the freedom
to give values in every ordinal, then the ``fan'' (i.e., the sequence
in $\ind_{f}$) which we got in Lemma \ref{lem:MustBeGood-1-1} might
not have been in a successor to a limit level, so we would have no
freedom in applying $G$ on it. As $\Gamma$ is relevant only for
limit levels, the coloring was defined only on sequence in $\Suc_{\lim}$,
so we needed $\eta$ to give elements from there. 
\end{rem}

\section{\label{sec:strongly dependent}Strongly dependent theories}

As we said in the introduction, in \cite{Sh863} it is proved that
$\beth_{\left|T\right|^{+}}\left(\lambda\right)\to\left(\lambda^{+}\right)_{T,n}$
for strongly dependent $T$ and $n<\omega$. 

In \cite{KaSh946} we show that in $RCF$ there is a similar phenomenon
to what we have here, but for $\omega$-tuples: there are sets from
all cardinalities with no indiscernible sequence of $\omega$-tuples
up to the first strongly inaccessible cardinal. This explains why
the theorem mentioned was only proved for $n<\omega$.

The example we described here is not strongly dependent, but it can
be modified a bit so that it will be, and then give a similar theorem
for strongly dependent theories (or even strongly$^{2}$ dependent),
but for $\omega$-tuples. 
\begin{thm}
For every $\theta$ there is a strongly$^{2}$ dependent theory $T$
of size $\theta$ such that for all $\kappa$ and $\delta$, \textup{$\kappa\to\left(\delta\right)_{T,\omega}$
iff $\kappa\to\left(\delta\right)_{\theta}^{<\omega}$. }\end{thm}
\begin{proof}
Right to left follows from Proposition \ref{prop:Easy direction}. 

For $n<\omega$ let $S_{n}=2^{\leq n}$ and let $T_{n}^{\theta}$
be the theory $T_{S_{n}}^{\theta}$ (see Corollary \ref{cor:TSThetaModelCompletion}).
Let $T$ be the theory $\sum_{n<\omega}T_{n}^{\theta}$: the language
is $\left\{ Q_{n}\left|\, n<\omega\right.\right\} \cup\left\{ R^{n}\left|\, R\in L_{S_{n}}^{\theta}\right.\right\} $
where $Q_{n}$ are unary predicates, and the theory says that they
are mutually disjoint and that each $Q_{n}$ is a model of $T_{n}^{\theta}$.
It is easy to see that this theory is complete and has quantifier
elimination. Denote $\Ss=2^{<\omega}$ as before. If $M$ is a model
of $T_{\Ss}^{\theta}$, then $M$ naturally induces a model $N$ of
$T$ (where $Q_{n}^{N}=\left(M\times\left\{ n\right\} \right)\upharpoonright L_{S_{n}}^{\theta}$).
For all $a\in M$, let $f_{a}\in\prod_{n<\omega}Q_{n}^{N}$ be defined
by $f_{a}\left(n\right)=\left(a,n\right)$ for $n<\omega$. Now, if
$A\subseteq P_{\left\langle \right\rangle }^{M}$ is any set with
no $\delta$-indiscernible sequence then the set $\left\{ f_{a}\left|\, a\in A\right.\right\} $
is a sequence of $\omega$-tuples with no indiscernible sequence of
length $\delta$. 

By Corollary \ref{cor:TSThetaModelCompletion}, it follows that each
$T_{n}$ is strongly$^{2}$ dependent, and so also $T$ (this can
be seen directly by Definition \ref{def:strongly2 dependent}, or
use an equivalent definition using mutually indiscernible sequences
\cite[Definition 2.3]{Sh863} and \cite[Claim 2.8 (3)]{Sh863}). 
\end{proof}
\bibliographystyle{alpha}
\bibliography{common}

\end{document}